\documentclass[12pt]{article}

\usepackage{amsmath}
\usepackage{amsfonts}
\usepackage{latexsym}
\usepackage{amssymb}
\usepackage{amsthm}
\usepackage{color}
\usepackage{csquotes}
\usepackage{bbding}
\usepackage[misc]{ifsym}
\usepackage[normalem]{ulem}

\makeatletter

\newtheorem{theorem}{Theorem}[section]
\newtheorem{proposition}[theorem]{Proposition}
\newtheorem{corollary}[theorem]{Corollary}
\newtheorem{lemma}[theorem]{Lemma}

\newtheorem{remark}[theorem]{Remark}

\newtheorem{example}{Example}
\newtheorem*{example*}{Example}
\def\1{\mathbf 1}

\def\ba{\mathbf a}

\def\bb{\mathbf b}

\def\I{\mathbf I}

\def\N{\mathbf N}

\def\x{\mathbf x}

\def\0{\mathbf 0}

\def\cB{\mathcal B}

\def\cD{\mathcal D}

\def\cG{\mathcal G}
\def\cH{\mathcal H}

\def\cL{\mathcal L}

\def\cP{\mathcal P}
\def\cX{\mathcal X}

\def\bZ{{\mathbb Z}}

\def\PG{{\rm PG}}

\def\GammaL{{\rm \Gamma L}}

\def\GF{{\rm GF}}
\def\GL{{\rm GL}}

\def\Gal{{\rm Gal}}
\def\I{{\rm I}}

\def\Fq{\mathbb F_{q}}
\def\Fqd{\mathbb F_{q^d}}
\def\Fqm{\mathbb F_{q^m}}

\def\Fqn{\mathbb F_{q^{n}}}

\def\Fqn{\mathbb F_{q^n}}
\def\Fqk{\mathbb F_{q^k}}
\def\Fqc{\mathbb F_{q^c}}
\def\Fqr{\mathbb F_{q^r}}

\def\Fp{\mathbb F_{p}}

\def\Aut{{\rm Aut}}

\def\diag{{\rm diag}}
\def\dim{{\rm dim}}

\def\End{{\rm End}}

\def\GF{{\rm GF}}
\def\gcd{{\rm gcd}}

\def\id{{\it id}}
\def\lcm{{\rm lcm}}

\def\Rad{{\rm Rad\,}}
\def\rank{{\rm rank}}

\def\Tr{{\rm Tr}}
\def\N{{\rm N}}

\def\sk{s^{(k)}}

\newcommand\comment[1]{}

\def\<{\langle}
\def\>{\rangle}

\title{Puncturing maximum rank distance codes}

\author{Bence Csajb\'ok \footnote{The first author is supported by the J\'anos Bolyai Research
Scholarship of the Hungarian Academy of Sciences. The first author
acknowledges the support of OTKA Grant No. K 124950.}\\
\small MTA--ELTE Geometric and Algebraic Combinatorics Research Group,
\\[-0.8ex]
\small ELTE E\"otv\"os Lor\'and University, Budapest, Hungary\\[-0.8ex] 
\small Department of Geometry\\[-0.8ex]
\small 1117 Budapest, P\'azm\'any P.\ stny.\ 1/C, Hungary\\
\small {\tt csajbokb@cs.elte.hu}\\
\and
Alessandro Siciliano\\
\small Dipartimento di Matematica, Informatica ed Economia\\[-0.8ex]
\small Universit\`a degli Studi della Basilicata\\[-0.8ex]
\small Potenza, Italy\\
\small{\tt alessandro.siciliano@unibas.it}
}

\begin{document}

%\date{\small{\em{version 3.6}}}
\date{}

\maketitle

\begin{abstract}
We investigate punctured maximum rank distance codes in cyclic models for bilinear forms of finite vector spaces.
In each of these models we consider an infinite family of linear maximum rank distance codes  obtained by puncturing generalized twisted Gabidulin codes. We calculate the automorphism group of such codes and we prove that this family contains many codes which are not equivalent to any generalized Gabidulin code. This solves a problem posed recently by  Sheekey in \cite{she}.

\end{abstract}


\section{Introduction}

Let $M_{m,n}(\Fq)$,  $m\le n$, be the rank metric space  of all the $m\times n$ matrices with entries in the finite field $\Fq$ with $q$ elements, $q=p^h$, $p$ a prime. The {\em distance} between two matrices  by definition  is the rank of their difference.
An  $(m,n,q;s)${\em-rank distance code} (also {\em rank metric code}) is any subset $\cX$ of $M_{m,n}(\Fq)$ such that the  distance between two of its distinct elements is at least  $s$.
An $(m,n,q;s)$-rank distance code is said to be {\em linear} if it is an $\Fq$-linear subspace of $M_{m,n}(\Fq)$.

  It is known \cite{del} that the size of an  $(m,n,q;s)$-rank distance code  $\cX$ is bounded by  the {\em Singleton-like bound}:
\[
|\cX| \le q^{n(m-s+1)}.
\]
When this bound is achieved, $\cX$ is called  an  $(m, n, q; s)${\em -maximum rank distance code}, or $(m, n, q; s)$-{\em MRD code} for short.

Although MRD  codes are very interesting by their own and they caught the attention of many researchers  in recent years \cite{alr,ckww,rav,she}, such codes also have practical applications in error-correction for random network
coding  \cite{gpt,kk,skk}, space-time coding \cite{tsc}  and  cryptography \cite{gpt2,sk}. 

Obviously, investigations of MRD codes can be carried out in any rank metric space isomorphic to $M_{m,n}(\Fq)$.
In his pioneering paper \cite{del}, Ph. Delsarte  constructed linear MRD codes for all the possible values of the parameters $m$, $n$, $q$ and  $s$ by using the framework  of  bilinear forms on two  finite-dimensional vector spaces over a finite  field. Delsarte called such sets {\em Singleton systems} instead of maximum rank distance codes.
Few years later, Gabidulin   \cite{gab} independently  constructed Delsarte's linear MRD codes as evaluation codes of linearized polynomials over a finite field \cite{ln}. Although originally discovered by Delsarte,  these codes are now called {\em Gabidulin  codes}.
In \cite{kg} Gabidulin's   construction was generalized to get different MRD codes. These codes  are now known as {\em Generalized Gabidulin codes}.
For $m=n$  a different construction of Delsarte's MRD codes was given by  Cooperstein \cite{coop}  in the framework of the tensor product of a vector space over $\Fq$ by itself.

Recently, Sheekey \cite{she} presented a new family of linear MRD codes by
using linearized polynomials over $\Fqn$. These codes are now known
as {\em generalized twisted Gabidulin codes}. The equivalence classes of
these codes were determined by Lunardon, Trombetti and Zhou in \cite{ltz}.
In \cite{prns} a further generalization was considered giving new MRD codes
when $m<n$; the authors call these codes generalized twisted Gabidulin
codes as well.
In this paper the term  "generalized twisted Gabidulin code" will be
used for codes defined in \cite[Remark 8]{she}.  For different relations between linear MRD codes and linear sets see \cite{bmpz,lun}, \cite[Section 5]{she}, \cite[Section 5]{cmp}. To the extent of our knowledge, these are  the only infinite families of linear MRD codes with $m<n$ appearing in the literature.

In \cite{ds}  infinite families of non-linear $(n,n,q;n-1)$-MRD codes, for $q\ge 3$ and $n\ge 3$ have been constructed. These families contain the  non-linear MRD  codes provided by   Cossidente, Marino and Pavese in \cite{cmp}. These codes have been   afterwards generalized in \cite{dondur}  by using a more geometric approach. 
A generalization of Sheekey's example which yields additive but not $\Fq$-linear codes can be found in \cite{oo}.

Let $\cX$ be a rank distance code in $M_{n,n}(\Fq)$. For any given $m\times n$ matrix $A$ over $\Fq$ of rank $m<n$, the set  $A\cX=\{AM: M\in \cX \}$ is a rank distance code in $M_{m,n}(\Fq)$. The code $A\cX$  is said to be obtained by {\em puncturing  $\cX$ with $A$} and  $A\cX$ is  called  a {\em punctured code}.   The reason of this definition is that if $A=(\I_m|\0_{n-m})$, where  $I_m$ and $\0_{n-m}$ is the $m\times m$ identity and $m\times(n-m)$ null matrix, respectively,   then the matrices of  $A\cX$ are obtained by deleting the last $n-m$ rows from the matrices in $\cX$. 
%In coding theory, the code  $A\cX$ is  known as  a {\em punctured code}.  
Punctured rank metric codes have been studied before in [3, 26] but
the equivalence problem among these codes have not been dealt with in
these papers.

In \cite[Remark 9]{she} Sheekey posed the following problem:  
\\[.1in]
{\em Are the MRD codes obtained by puncturing generalized twisted Gabidulin codes equivalent to the codes obtained by puncturing generalized Gabidulin codes?}

Here we investigate punctured codes and study the above problem in the framework of bilinear forms. We point out that the very recent preprint \cite{tz} deals with the same problem by using $q$-linearized polynomials. In \cite{tz} the authors  investigate the middle nucleus and the right nucleus of  punctured generalized twisted Gabidulin codes, for $m<n$. By exploiting these nuclei, they derive necessary conditions on the automorphisms of these codes  which depend on certain restrictions for the parameters.

Let $V$ and $V'$ be  two vector spaces over $\Fq$ of dimensions $m$ and $n$, respectively. Since the rank is invariant under matrix transposition, we may assume $m\le n$.

A {\em bilinear form} on $V$ and $V'$ is a function  $f:V\times V'\rightarrow  \Fq$
that satisfies the identity
\[
f\left(\sum_{i}{x_i v_i},\sum_{j}{x'_j v'_j}\right)=\sum_{i,j}{x_if(v_i,v'_j)x_j'},
\]
for all scalars $x_i,x'_j\in \Fq$ and all vectors $v_i\in V$, $v'_j\in V'$.
 The set  $\Omega_{m,n}=\Omega(V,V')$ of all bilinear forms on $V$ and $V'$ is  an  $mn$-dimensional vector space over $\Fq$.
 
The {\em left radical } $\Rad(f)$ of any $f\in\Omega_{m,n}$ is by definition  the subspace of $V$ consisting of all vectors $v$ satisfying $f(v,v')=0$ for every $v'\in V'$. The {\em rank} of $f$ is the codimension of $\Rad(f)$, i.e.
\begin{equation}\label{eq_34}
\rank(f)=m-\dim_{\Fq}(\Rad(f)).
\end{equation}
Then the $\Fq$-vector space $\Omega_{m,n}$ equipped with the above rank function is a rank metric space over $\Fq$.

Let $\{u_0,\ldots, u_{m-1}\}$  and  $\{u'_0,\ldots, u'_{n-1}\}$ be a basis for $V$ and $V'$, respectively. For any  $f\in\Omega_{m,n}$,  the $m\times n$ $\Fq$-matrix   $M_f=(f(u_i,u'_j))$, is called the {\em matrix of} $f$ {\em in the bases} $\{u_0,\ldots, u_{m-1}\}$ {\em and}  $\{u'_0,\ldots, u'_{n-1}\}$. It turns out that  the map
\begin{equation}\label{eq_4}
\begin{array}{rccc}
\nu_{\{u_0,\ldots, u_{m-1};u'_0,\ldots, u'_{n-1}\}}:& \Omega_{m,n} & \rightarrow &  M_{m,n}(\Fq)\\
 & f & \mapsto  & M_f
\end{array}
\end{equation}
 is an isomorphism of rank metric spaces with $\rank(f)=\rank(M_f)$.

{Let  $\GammaL(\Omega_{m,n})$ denote the {\em general semilinear group} of the  $mn$-dimensional $\Fq$-vector space  $\Omega_{m,n}$, that  is, the group of all invertible semilinear transformations of $\Omega_{m,n}$. 
 Let $\{w_1,\ldots,w_{mn}\}$ be a basis for $\Omega_{m,n}$, and recall that $\Aut(\Fq)=\<\phi_p\>$, where $\phi_p:\Fq\rightarrow\Fq$ is the Frobenius map $\lambda\mapsto\lambda^p$. Using $\phi_p$, we define the map $\phi:\Omega_{m,n}\rightarrow\Omega_{m,n}$ by
\[
\phi:\sum_{i}{\lambda_i w_i}\mapsto \sum_{i}{\lambda_i^p w_i}.
\]
Then $\phi$ is an invertible semilinear transformation of $\Omega_{m,n}$, and for $(a_{ij})\in\GL(mn,q)$ we have $(a_{ij})^\phi=(a_{ij}^p)$. Therefore $\phi$ normalizes the general linear group $\GL(mn,q)$ and we have $\GammaL(\Omega_{m,n})=\GL(\Omega_{m,n}) \rtimes \Aut(\Fq)$. 
\\
An {\em automorphism} of the rank metric space $\Omega_{m,n}$ is any transformation $\tau\in\GammaL(\Omega_{m,n})$ such that $\rank(f^\tau)=\rank(f)$, for all $f\in\Omega_{m,n}$.  The {\em automorphism group} $\Aut(\Omega_{m,n})$ of $\Omega_{m,n}$  is the group  of all automorphisms of $\Omega_{m,n}$, i.e.
\[
\Aut(\Omega_{m,n})=\{\tau\in\GammaL(\Omega_{m,n}):\rank(f^\tau)=\rank(f),\ \mathrm{ for \ all\ } f\in\Omega_{m,n}\}.
\]

By \cite[Theorem 3.4]{wan}, 
\[
\mbox{$\Aut(\Omega_{m,n})=(\GL(V)\times\GL(V')) \rtimes\Aut(\Fq)$ for $m<n$},
\]
 and 
 \[
 \mbox{$\Aut(\Omega_{n,n})=(\GL(V')\times\GL(V'))\rtimes \langle \top\rangle \rtimes\Aut(\Fq)$ for $m=n$},
 \]
  where $\top$ is an involutorial operator.
 In details, any given $(g,g')\in\GL(V)\times \GL(V')$ defines the linear automorphism of $\Omega_{m,n}$ given by 
\[
 f^{(g,g')}(v,v')=f(gv,g'v'),
 \]
for any $f\in\Omega_{m,n}$. If $A$ and $B$ are the matrices of $g\in\GL(V)$ and $g'\in\GL(V')$ in the given bases for $V$ and $V'$,  then the matrix of  $f^{(g,g')}$ is $A^tM_fB$, where $t$ denotes transposition. 
Additionally, the  semilinear transformation $\phi$ of $\Omega_{m,n}$ is the automorphism  given by 
\[
f^{\phi}(v,v')=[f(v^{\phi^{-1}},{v'}^{\phi^{-1}})]^p.
\] 
If  $M_f=(a_{ij})$ is the  matrix of $f$ in the given bases for $V$ and $V'$, then the matrix of $f^{\phi}$ is $M_f^{\phi}=(a_{ij}^{p})$. Therefore $\phi$ normalizes the group $\GL(V)\times\GL(V')$. If $m<n$, the above automorphisms are all the elements in  $\Aut(\Omega_{m,n})$. 
\\
If  $m=n$, one may assume, and we do,  $V'=V=\langle  u_0,\ldots, u_{m-1}\rangle$.  The involutorial operator $\top:\Omega_{n,n}\rightarrow \Omega_{n,n}$   is defined by setting 
 \[
 f^\top(v,v')=f(v',v).
 \]
 If  $M_f=(a_{ij})$ is the  matrix of $f$ in the given bases for $V$ and $V'$, then the matrix of $f^{\top}$ is the transpose matrix $M_f^t$ of $M_f$. The operator $\top$ acts on $\GL(V)\times\GL(V)$ by mapping $(g,g')$ to $(g',g)$.

For a given subset $\cX$ of $\Omega_{m,n}$, the {\em automorphism group} of $\cX$  is the subgroup of $\Aut(\Omega_{m,n})$ fixing $\cX$. Two subsets  $\cX_1,\cX_2$ of $\Omega_{m,n}$ are said to be {\em equivalent} if there exists $\varphi\in \Aut(\Omega_{m,n})$ such that $\cX_2=\cX_1^\varphi$.

%The paper is organized as follows.
The main tool we use in this paper is  the $k$-cyclic model in $V(r,q^r)$ for an $r$-dimensional vector space  $V(r,q)$ over $\Fq$, where $k$ is any positive integer such that $\gcd(r,k)=1$. This model generalizes the cyclic model  introduced in \cite{coop,fkmp,h1} and it  is studied in Section \ref{sec_2}.  In particular, the endomorphisms of the  $k$-cyclic model are represented by  $r\times r$ $q^k$-circulant matrices over $\Fqr$. 
\\
For any $k$ such that $\gcd(m,k)=1=\gcd(n,k)$,  the elements of $\Omega_{m,n}$ acting on the $k$-cyclic model of $V$ and $V'$ are represented by $q^k$-circulant $m\times n$ matrices over $\Fqd$, where  $d=\lcm(m,n)$. We then have a description of the elements in $\Aut(\Omega_{m,n})$ in terms of $q^k$-circulant matrices.
\\
In Section \ref{sec_3} we prove that the code obtained by puncturing an $(n,n,q;s)$-MRD code is an $(m,n,q;s+m-n)$-MRD code, where $n-s<m\le n$. In particular, the code in $\Omega_{m,n}$ obtained by puncturing a generalized Gabidulin code in $\Omega_{n,n}$ is a generalized Gabidulin code. Conversely, every  generalized Gabidulin code in $\Omega_{m,n}$ can be obtained by puncturing a  generalized Gabidulin code in $\Omega_{n,n}$.  \\
By using  the representation by $q^k$-circulant matrices of the elements of $\Omega_{m,n}$ acting on the $k$-cyclic model for $V$ and $V'$, we calculate the automorphism group of some generalized Gabidulin code.
In Section \ref{sec_3} we  also construct an infinite family of MRD codes by puncturing generalized twisted Gabidulin codes \cite{she,ltz}. 
 We calculate the automorphism group of these  codes in Section \ref{sec_4}. By using a recent result by Liebhold and  Nebe  \cite{lnebe}, we prove in  Section \ref{sec_5} that  the above family contains many  MRD codes which are inequivalent to the MRD codes obtained by puncturing generalized Gabidulin codes. This solves the problem posed by Sheekey in \cite[Remark 9]{she}.

%In this case,  we   determine  completely the automorphism group of  punctured code obtained from generalized Gabidulin codes  in Section \ref{sec_4}. Finally, in Section  \ref{sec_5} we determine  the automorphism group of some punctured  generalized twisted Gabidulin codes.  By using a recent result by Liebhold and  Nebe  \cite{lnebe}, we can prove that  these MRD codes  are inequivalent to the MRD obtained by puncturing  generalized Gabidulin codes. %
%
%
%%%%%%%%%%%%%%%%%%%%%%%%%%%%%%%%%%%%%%%%%%%%%%%%%%%%%%%%%%%%%%%%%%%%%%%%%
%%%%%%%%%%%%%%%%%%%%%%%%%%%%%%%%%%%%%%%%%%%%%%%%%%%%%%%%%%%%%%%%%%%%%%%%%
%%%%%%%% The cyclic representation of bilinear forms     %%%%%%%%%%%%%%%%
%%%%%%%%         on finite vector spaces                 %%%%%%%%%%%%%%%%
%%%%%%%%%%%%%%%%%%%%%%%%%%%%%%%%%%%%%%%%%%%%%%%%%%%%%%%%%%%%%%%%%%%%%%%%%
%%%%%%%%%%%%%%%%%%%%%%%%%%%%%%%%%%%%%%%%%%%%%%%%%%%%%%%%%%%%%%%%%%%%%%%%%
%%%%%%%%%%%%%%%%%%%%%%%%%%%%%%%%%%%%%%%%%%%%%%%%%%%%%%%%%%%%%%%%%%%%%%%%%
%
%
\section{Cyclic models for bilinear forms on finite vector spaces}\label{sec_2}

Let $V(r,q)=\<u_0,\ldots, u_{r-1}\>_{\Fqr}$, $r\ge 2$, be an $r$-dimensional vector space over the finite field $\Fqr$. We  denote the set of all linear transformations of $V(r,q)$   by $\End(V(r,q))$.

 Embed $V(r,q)$ in $V(r,q^r)$ by extending the scalars. Concretely this can be done by defining $V(r,q^r)=\{\sum_{i=0}^{r-1}{\lambda_i u_i:\lambda_i \in \Fqr}\}$.

Let $\xi:V(r,q^r)\rightarrow V(r,q^r)$ be the $\Fqr$-semilinear transformation with associated automorphism $\delta:x\in \Fqr\rightarrow x^q\in\Fqr$ such that $\xi(u_i)=u_i$. Clearly, $V(r,q)$ consists of all the vectors in $V(r,q^r)$ which are fixed by $\xi$.

In the paper \cite{coop},  the cyclic model of  $V(r,q)$   was introduced by taking the  eigenvectors $s_0,\ldots, s_{r-1}$ in $V(r,q^r)$ of a  Singer cycle $\sigma$ of $V(r,q)$; here a {\em Singer cycle} of $V(r,q)$ is an element $\sigma$ of $\GL(V(r,q))$ of order $q^r-1$. The cyclic group $S=\langle \sigma\rangle $ is called a {\em Singer cyclic group of} $\GL(V(r,q))$.

Since  $s_0,\ldots, s_{r-1}$ have distinct eigenvalues in $\Fqr$, they form a basis of  the extension $V(r,q^r)$ of $V(r,q)$.  

 In this basis the matrix of $\sigma$   is the diagonal matrix $\diag(w,w^q,\ldots,w^{q^{r-1}})$, where $w$ is a primitive element of $\Fqr$ over $\Fq$ and $w^{q^i}$ is the eigenvalue of $s_i$. The action of the  linear part $\ell_{\xi}$ of the $\Fqr$-semilinear transformation $\xi$ is given by $\ell_{\xi}(s_i)=s_{i+1}$,  where the indices are considered modulo $r$ \cite{coop}.   It follows that 
\begin{equation}\label{eq_1_1}
V(r,q)=\left\{\sum_{i=0}^{r-1}{x^{q^{i}}s_i}:x \in \Fqr\right\}.
\end{equation}
%Note that $\langle \ell_\xi\rangle $ normalizes $S$ in $\GL(V(r,q))$ and $S\rtimes \langle \ell_\xi\rangle $ is the normalizer  in $\GL(V(r,q))$ of  $S$ \cite{hu}. 
We call $\{s_0,\ldots, s_{r-1}\}$ a {\em Singer basis} for $V(r,q)$ and the representation (\ref{eq_1_1}) for $V(r,q)$, or equivalently the set $\{(x,x^q,\ldots, x^{q^{r-1}}):x \in\Fqr\}\subset\Fqr^r$, is  the {\em cyclic model for} $V(r,q)$ \cite{fkmp,h1}.

We point out that the $\Fqr$-semilinear transformation $\phi:V(r,q^r)\rightarrow V(r,q^r)$ with associated automorphism the Frobenius map $\phi_p:x\in \Fqr\rightarrow x^p\in\Fqr$ such that $\phi(u_i)=u_i$ acts on the cyclic model (\ref{eq_1_1}) by mapping $xs_0+x^qs_1+...+ x^{q^{r-1}}s_{r-1}$ to $x^{pq^{r-1}}s_0+x^ps_1+...+ x^{pq^{r-2}}s_{r-1}$.

%The lattice of $\Fq$-subspaces spanned by vectors in the cyclic model for $V(r,q)$ is called the {\em cyclic model for  the projective space} $\PG(V,\Fq)$.

%An $r\times r$ $q$-{\em circulant} (or  {\em Dickson } ) {\em matrix} over $\Fqr$ is a matrix of the form
% \[
%D_{(a_0,a_1,\ldots,a_{r-1})}= \begin{pmatrix}
%a_0           & a_1           & \cdots  & a_{r-1}\\
%a_{r-1}^q     & a^q_{0}       & \cdots  & a_{r-2}^q\\
%\vdots        & \vdots        &  \ddots & \vdots\\
%a^{q^{r-1}}_{1} & a^{q^{r-1}}_{2} & \cdots  & a^{q^{r-1}}_0
%\end{pmatrix}
% \]
%with $a_i\in \Fqr$. We say that the above matrix is {\em generated  by the array} $(a_0,\ldots,a_{r-1})$. Let $\cD_r(\Fqr)$ denote the {\em Dickson matrix algebra} formed by all $r\times r$ $q$-circulant matrices over $\Fqr$. The set $\cB_r(\Fqr)$ of all invertible $q$-circulant  $r\times r$ matrices is known as the {\em Betti-Mathieu group} \cite{car}.
%
%

Let $k$ be a positive integer such that $\gcd(k,r)=1$. Set $\sk_i=s_{ki \bmod r}$, for $i=0,\ldots, r-1$. For brevity, we use $[j]=q^{j}$ and $a^{[j]}=a^{q^{j}}$, for any $a\in\Fqr$. It is clear that the exponent $j$ is taken mod $r$ because of the field size. Then we may write
\begin{equation}\label{eq_16}
V(r,q)=\left\{\sum_{i=0}^{r-1}{x^{[ki]}\sk_i}:x \in \Fqr\right\}.
\end{equation}

%It is easily seen that  $V(r,q)$ as given in (\ref{eq_16})  is pointwise fixed by the semilinear transformation $\xi^k$ defined by
%\begin{equation}\label{eq_1}
%\xi^k(\sk_i)=\sk_{i+1}, \ \ 0\le i< r-1, \ \  \xi(\sk_{r-1})=\sk_{0}
%\end{equation}
%with companion automorphism $\delta^k:x\in\Fqr\mapsto x^{q^k}\in\Fqr$.

We call the representation (\ref{eq_16}) for $V(r,q)$, or equivalently the set $\{(x,x^{[k]},\ldots, x^{[k(r-1)]}):x \in\Fqr\}\subset\Fqr^r$,    the $k${\em-cyclic model for} $V(r,q)$.  %The lattice of $\Fq$-subspaces spanned by vectors in  (\ref{eq_16}) is called the $k${\em-th cyclic model for  the projective space} $\PG(V,\Fq)$.

 It is easily seen that the linear part of the semilinear transformation $\xi^k$ acts on the $k$-th cyclic model for $V(r,q)$ by mapping $s_i^{(k)}$ to $s_{i+1}^{(k)}$, with indices considered modulo $r$.

  An $r\times r$ $q^k$-{\em circulant}  {\em matrix} over $\Fqr$ is a matrix of the form
 \[
D_{(a_0,a_1,\ldots,a_{r-1})}^{(k)}= \left(\begin{array}{cccc}
a_0     & a_1           & \cdots  & a_{r-1}\\
a_{r-1}^{[k]}     & a_{0}^{[k]}       & \cdots  & a_{r-2}^{[k]}\\
\vdots        & \vdots        &  \ddots & \vdots\\
a_{1}^{[k(r-1)]} & a_{2}^{[k(r-1)]} & \cdots  & a_0^{[k(r-1)]}
\end{array}\right)
 \]
with $a_i\in \Fqr$. We say that the above matrix is {\em generated  by the array} $(a_0,\ldots,a_{r-1})$.

Let $\cD_r^{(k)}(\Fqr)$ denote the matrix algebra formed by all $r\times r$ $q^k$-circulant matrices over $\Fqr$ and  $\cB_r^{(k)}(\Fqr)$ the set  of all invertible $q^k$-circulant  $r\times r$ matrices.
When $k=1$, an $r\times r$ $q$-circulant matrix  over $\Fqr$ is also known as a {\em Dickson matrix}, $\cD_r(\Fqr)=\cD_r^{(1)}(\Fqr)$ is the {\em Dickson matrix algebra}  and  $\cB_r(\Fqr)=\cB_r^{(1)}(\Fqr)$ is the {\em Betti-Mathieu group} \cite{bott,car}.
It is known that $\End(V(r,q))\simeq\cD_r(\Fqr)$ and $\cB_r(\Fqr)\simeq\GL(V(r,q))$ 	\cite{ln,wl}.

\begin{remark}\label{rem_6}
{\em 
In terms of matrix representation, the above isomorphisms  are described as follows.
Let $V(r,q)=\<u_0,\ldots, u_{r-1}\>_{\Fq}$ and $\{s_0,\ldots, s_{r-1}\}$ a Singer basis for $V(r,q)$ defined by the primitive element $w$ of $\Fqr$ over $\Fq$. Up to a change of the  basis $\{u_0,\ldots, u_{r-1}\}$ in $V(r,q)$,  we may assume
\[
u_i=w^{i}s_0+\ldots+w^{iq^{r-1}}s_{r-1}, \ \mathrm{for\ }i=0,\ldots, r-1.
\]
Notice that $u_i\in\ V(r,q)$, for  $i=0,\ldots, r-1$. The non-singular Moore matrix
\begin{equation}\label{eq_24}
E_{r}=\left(\begin{array}{cccc}
1 & w & \cdots & w^{r-1}\\
1 & w^q & \cdots & w^{(r-1)q}\\
\vdots  &   \vdots & & \vdots\\
1 & w^{q^{r-1}} & \cdots & w^{(r-1)q^{r-1}}
\end{array}\right)
\end{equation}
 is the matrix of the change of basis from $\{u_0,\ldots, u_{r-1}\}$ to $\{s_0,\ldots, s_{r-1}\}$. Therefore, the matrix map $D\in\cD_{r}(\Fqr)\rightarrow E_r^{-1}DE_r\in M_{r,r}(\Fq)$ realizes the above isomorphism.
}
\end{remark}

\begin{proposition}\label{prop_3}
$\End(V(r,q))\simeq\cD_r^{(k)}(\Fqr)$ and $\GL(V(r,q))\simeq \cB_r^{(k)}(\Fqr)$.
\end{proposition}
\begin{proof}
For any $\ba=(a_0,\ldots, a_{r-1})$ over  $\Fqr$, the $q^k$-circulant matrix  $D_{\ba}^{(k)}$ acts on the  $k$-th cyclic model \eqref{eq_16} for $V(r,q)$ by mapping $(x,x^{[k]},\ldots, x^{[k(r-1)]})$ to $(a_0x+a_1x^{[k]}+\ldots+ a_{r-1}x^{[k(r-1)]},a_{r-1}^{[k]}x+a_0^{[k]}x^{[k]}+\ldots+ a_{r-2}^{[k]}x^{[k(r-1)]},\ldots, a_{1}^{[k(r-1)]}x+a_2^{[k(r-1)]}x^{[k]}+\ldots+ a_{0}^{[k(r-1)]}x^{[k(r-1)]})$, giving $D_{\ba}^{(k)}$ is an endomorphism of \eqref{eq_16}. Let $D_\ba, D_{\ba'}\in\cD_r^{(k)}(\Fqr)$ such that $D_\ba\x^t= D_{\ba'}\x^t$, for every $\x=(x,x^{[k]},\ldots, x^{[k(r-1)]})$, $x \in \Fqr$. Hence, $(a_0-a'_0)x+(a_1-a'_1)x^{[k]}+\ldots+ (a_{r-1}-a'_{r-1})x^{[k(r-1)]}=0$, for all $x\in \Fqr$. As the left hand side is a polynomial of degree at most $q^{r-1}$ with $q^r$ roots, we get $\ba=\ba'$. Therefore,  matrices in   $\cD_r^{(k)}(\Fqr)$ represent $q^{r^2}$ distinct endomorphisms of the  $k$-th cyclic model  for $V(r,q)$.  As $q^{r^2}=|\End(V(r,q))|$, we get the result.
 % \qed
\end{proof} 
%
%\begin{remark}
%Let $\cL_^{(k)}$From the proof of the previous Proposition we see that the map 
%\end{remark}
%
\begin{remark}\label{rem_7}
{\em
  Let $K_r$ be the (permutation) matrix  of the change of basis from  $\{s_0^{(k)},\ldots,{s_{r-1}}^{(k)}\}$ to $\{s_0,\ldots, s_{r-1}\}$. As $\sk_i=s_{ik \bmod r}$, for $i=0,\ldots, r-1$, then the $i$-th column of $K_r$ is the  array $(0,\ldots,0,1,0,\ldots,0)^t$ where 1 is in position $ik \bmod r$, for $i=0,\ldots,r-1$. If $\tau\in\End(V(r,q))$ has $q^k$-circulant matrix $D_{(a_0,a_1,\ldots,a_{r-1})}^{(k)}$ in the basis  $\{s_0^{(k)},\ldots,{s_{r-1}}^{(k)}\}$, then the matrix of $\tau$ in the Singer basis $\{s_0,\ldots,s_{r-1}\}$ is the $q$-circulant matrix  $D_{(b_0,\ldots,b_{r-1})}=K_rD_{(a_0,a_1,\ldots,a_{r-1})}^{(k)}K_r^{-1}$, for some array $(b_0,\ldots,b_{r-1})$ over $\Fqr$.
 Since $\gcd(k,r)=1$, we can write $1=lr+hk$, for some integers $l,h$, giving
 \[
 b_i=a_{ih \bmod r}, \ \ \ \, \mathrm{for \ } i=0,\ldots, r-1.
 \]
 Therefore, $\cD_r^{(k)}(\Fqr)=K_r^{-1}\cD_r(\Fqr)K_r$ and $\cB_r^{(k)}(\Fqr)=K_r^{-1}\cB_r(\Fqr)K_r$.
}
 \end{remark}
 \begin{remark}\label{rem_9}
 {\em We explicitly describe the action of $\Aut(\Fq)$ on $V(r,q^r)$ in the Singer basis $\{s_0^{(k)},\ldots,s_{r-1}^{(k)}\}$. By Remark \ref{rem_6} the invertible semilinear transformation $\phi$ of $V(r,q^r)$ defined by the Frobenius map $\phi_p:x\in \Fqr\rightarrow x^p\in\Fqr$ acts in the basis $\{s_0,\ldots,s_{r-1}\}$ via the pair $(E_r(E_r^{-1})^p;\phi_p)$, where $E_r$ is the non-singular Moore matrix (\ref{eq_24}) and $(E_r^{-1})^p$ is the matrix obtained by $E_r^{-1}$ by applying $\phi_p$ to every entry. By Remark \ref{rem_7} $\phi$  acts in the basis $\{s_0^{(k)},\ldots,s_{r-1}^{(k)}\}$ via the pair $(K_r^{-1}E_r(E_r^{-1})^pK_r;\phi_p)$, since $K_r^p=K_r$.}
 \end{remark}

Let $V=\langle  u_0,\ldots, u_{m-1}\rangle_{\Fq}$ and $V'=\langle u'_0,\ldots, u'_{n-1}\rangle_{\Fq}$, with $m\le n$. If $m=n$ we take $V'=V=\langle  u_0,\ldots, u_{m-1}\rangle_{\Fq}$.  Let   $\sigma$  and $\sigma'$ be  Singer cycles of $\GL(V)$ and $\GL(V')$, respectively,  with  associated semilinear transformations $\xi$ and $\xi'$. Let $\{s_0,\ldots,s_{m-1}\}$  and $\{s'_0,\ldots,s'_{n-1}\}$ be a Singer basis for $V$  and $V'$,   defined by $\sigma$ and $\sigma'$, respectively.  For any given positive integer $k$ such that $\gcd(k,n)=\gcd(k,m)=1$, let $\{\sk_0,\ldots,\sk_{m-1}\}$ and $\{s_0'^{(k)},\ldots,s_{n-1}'^{(k)}\}$ be the  bases of $V(m,q^m)$ and $V(n,q^n)$ defined as above. Therefore, we may consider $\Omega_{m,n}$ as the set of all bilinear forms acting on the $k$-th cyclic model for $V$ and $V'$.  In addition, any element in $\GL(V)\times\GL(V')$   is represented by a pair $(A,B)\in\cB_m^{(k)}(\Fqm)\times \cB_{n}^{(k)}(\Fqn)$.

 Set  $e=\gcd(m,n)$ and  $d=\lcm(m,n)$,  the greatest  common divisor and  the least common multiple  of $m$ and $n$, respectively. % and write $m=en$, $n=en'$.

Let  $\Tr_{q^d/q}$ denote  the trace function from $\Fqd$ onto $\Fq$:
\[
\Tr_{q^d/q}:y  \in \Fqd \rightarrow \Tr_{q^d/q}(y)=\sum_{i=0}^{d-1}{y^{q^i}}\in \Fq.
\]

Since $\gcd(k,d)=1$,  we may write $\Tr_{q^d/q}$ as
\[
T^{(k)}:y  \in \Fqd \rightarrow T^{(k)}(y)=\sum_{i=0}^{d-1}{y^{[k]}}\in \Fq.
\]

 For $0\le j\le e-1$ and a given $a\in\Fqd$ and $v=x\sk_0+\ldots+x^{[k(m-1)]}\sk_{m-1}\in V$ and $v'=x's_0'^{(k)}+\ldots+{x'}^{[k(n-1)]}s_{n-1}'^{(k)}$, the map
 \begin{equation}\label{eq_6}
f_{a,j}^{(k)}(v,v')=T^{(k)}(a x x'^{[kj]})
\end{equation}
 is a bilinear form on the $k$-cyclic model for $V$ and  $V'$. We set
\begin{equation}\label{eq_28}
\Omega_{j}^{(k)}=\{f_{a,j}^{(k)}:a\in\Fqd\}, \ \ \ \mathrm{ for\ } 0\le j \le e-1.
\end{equation}
The following result gives the decomposition of $\Omega_{m,n}$ as sum of the subspaces $\Omega_{j}^{(k)}$.
\begin{theorem}
\begin{equation}\label{eq_7}
\Omega_{m,n}=\bigoplus_{j=0}^{e-1}{\Omega_{j}^{(k)}}.
\end{equation}
\end{theorem}
\begin{proof} 
 Let first assume $k=1$. For any $e$-tuple  $\ba=(a_0,\ldots,a_{e-1})$ over $\Fqd$ we define an $m\times n$ matrix $D_{\ba}=D_{\ba}^{(1)}=(d_{i,j})$ over $\Fqd$ as follows.
We will use indices from 0 for both rows and columns of $D$. Let $d_{0,j}=a_j$, for $0\le j\le e-1$, and let $d_{i,j}=d_{i-1,j-1}$, where the row index is taken  modulo $m$ and the column index is  taken modulo $n$. Notice that the above rule determines every entry of $D_{\ba}$. In fact, $d_{i,j}=a_l^{q^s}$, where  $l\equiv j-iv\pmod e$, $0\le l\le e-1$ and $s=\beta m+i$, where $\beta$ is the unique integer in $\{0,1,\ldots,n/e-1\}$ such that $j-i \equiv l+\beta m \pmod n$.
\\
Now let $f_{a,j}\in\Omega_{j}$. Then the matrix of $f_{a,j}$ in the Singer bases $\{s_0,\ldots, s_{m-1}\}$ and $\{s_0',\ldots, s_{n-1}'\}$ is the matrix obtained by applying the above construction to the array $\ba=(0,\ldots,0,a,0,\ldots,0)$, with $a$ in the $j$-th position. It is now easy to see that the $\Fq$-spaces $\Omega_{j}$, for $j=0,\ldots,e-1$ intersect trivially. 
By  consideration on dimensions  we may write $\Omega_{m,n}=\bigoplus_{j=0}^{e-1}{\Omega_{j}}$.\\
The $k$-cyclic model for $V'$ and $V$ is obtained from the 1-cyclic model by applying the  changing of basis described in Remark  \ref{rem_7}. Therefore the $\Fq$-spaces $\Omega_{j}^{(k)}$, $k>1$, are pairwise skew and $\Omega_{m,n}=\bigoplus_{j=0}^{e-1}{\Omega_{j}^{(k)}}$.
 \end{proof}

\begin{example}
{\em
Let $m=2$, $n=6$ and $k=1$, so that $d=6$ and $e=2$. For any array $\ba=(a_0,a_1)$ over $\mathbb F_{q^{6}}$, we have
\[
D_\ba=\left(\begin{array}{cccccc}
a_0&a_1&a_0^{q^2}&a_1^{q^2}&a_0^{q^4}&a_1^{q^4}\\
a_1^{q^5}&a_0^q&a_1^q&a_0^{q^3}&a_1^{q^3}&a_0^{q^5}
\end{array}\right).\]
}
\end{example}
\begin{example}
{\em 
Let $m=4$, $n=6$ and $k=5$, so that $d=12$ and $e=2$. For any array $\ba=(a_0,a_1)$ over $\mathbb F_{q^{12}}$, we have
\[
D_{\ba}^{(k)}=\left(\begin{array}{cccccc}
a_0&a_1&a_0^{[8k]}&a_1^{[8k]}&a_0^{[4k]}&a_1^{[4k]}\\
a_1^{[5k]}&a_0^{[k]}&a_1^{[k]}&a_0^{[9k]}&a_1^{[9k]}&a_0^{[5k]}\\
a_0^{[6k]}&a_1^{[6k]}&a_0^{[2k]}&a_1^{[2k]}&a_0^{[10k]}&a_1^{[10k]}\\
a_1^{[11k]}&a_0^{[7k]}&a_1^{[7k]}&a_0^{[3k]}&a_1^{[3k]}&a_0^{[11k]}
\end{array}\right)=
\left(\begin{array}{cccccc}
a_0&a_1&a_0^{q^{4}}&a_1^{q^4}&a_0^{q^8}&a_1^{q^8}\\
a_1^{q}&a_0^{q^5}&a_1^{q^5}&a_0^{q^9}&a_1^{q^9}&a_0^q\\
a_0^{q^6}&a_1^{q^6}&a_0^{q^{10}}& a_1^{q^{10}}&a_0^{q^2}&a_1^{q^2} \\
a_1^{q^7}&a_0^{q^{11}}&a_1^{q^{11}}&a_0^{q^3}&a_1^{q^3}&a_0^{q^{7}}
\end{array}\right).\]
}
\end{example}

We call a matrix of type $D_{\ba}^{(k)}$  an $m\times n$ $q^k${\em-circulant}   {\em matrix}  over $\Fqd$, where $d=\lcm(m,n)$.  We  say that $D_\ba^{(k)}$ is {\em generated  by the array} $\ba=(a_0,a_1,\ldots,a_{e-1})$, where $e=\gcd(m,n)$. We will  denote  the set of all $m\times n$ $q^k$-circulant matrices over $\Fqd$ by $\cD_{m,n}^{(k)}(\Fqd)$.

The next result gives a description of $\Omega_{m,n}$ and $\Aut(\Omega_{m,n})$ in terms of $q^k$-circulant matrices.

\begin{proposition}\label{prop_8}
Let $m\le n$. Then $\Omega_{m,n}\simeq \cD_{m,n}^{(k)}(\Fqd)$.

If $m<n$, then
\[
\Aut(\Omega_{m,n})\simeq (\cB_m^{(k)}(\Fqm)\times\cB_n^{(k)}(\Fqn))\rtimes\Aut(\Fq);
\]
if $m=n$, then
\[
\Aut(\Omega_{n,n})\simeq (\cB_n^{(k)}(\Fqm)\times\cB_n^{(k)}(\Fqn))\rtimes\langle \top\rangle \rtimes\Aut(\Fq).
\]
\end{proposition}

\begin{proof}
For any $\ba=(a_0,\ldots, a_{e-1})$ over  $\Fqd$ we consider the bilinear form $f_\ba^{(k)}=f_{a_0,0}^{(k)}+\ldots +f_{a_{e-1},e-1}^{(k)}$. Straightforward calculation shows that the matrix of $f_\ba^{(k)}$ in  the  bases $\{\sk_0,\ldots,\sk_{m-1}\}$ and $\{s_0'^{(k)},\ldots,s_{n-1}'^{(k)}\}$  is the $m\times n$ $q^k$-circulant matrix $D_\ba^{(k)}$ generated by $\ba$.   Now assume that $f_\ba^{(k)}$ is the null bilinear form.  Let $V=\<u_0,\ldots, u_{m-1}\>_{\Fq}$ and $V'=\<u_0',\ldots, u_{n-1}'\>_{\Fq}$. By Remarks \ref{rem_6} and \ref{rem_7} the matrix of $f_\ba^{(k)}$ in the bases $\{u_0,\ldots, u_{m-1}\}$ and $\{u_0',\ldots, u_{n-1}'\}$  is  $(K_m^{-1}E_m)^{t} D^{(k)}_\ba (K_n^{-1}E_n)$, which is clearly the zero matrix. As $K_m^{-1}E_m$ and  $K_n^{-1}E_n$ are both non singular we get  $D_\ba^{(k)}$ is the zero matrix giving $\ba$ is the zero array.
Therefore,  matrices in  $\cD_{m,n}^{(k)}(\Fqr)$ represent $q^{de}=q^{mn}$ distinct bilinear forms  acting  on the  $k$-th cyclic models  for $V$ and $V'$. As  $q^{mn}=|\Omega_{m,n}|$, we get  $\Omega_{m,n}\simeq \cD_{m,n}^{(k)}(\Fqd)$. 

To prove the second part of the Proposition we first note that Proposition \ref{prop_3} implies that the group of all $\Fq$-linear automorphisms of $\Omega_{m,n}$ is isomorphic to $(\cB_m^{(k)}(\Fqm)\times\cB_n^{(k)}(\Fqn))$, if $m<n$, and to $(\cB_m^{(k)}(\Fqm)\times\cB_n^{(k)}(\Fqn))\rtimes\langle \top\rangle$, if $m=n$.
\\
%The the semilinear automorphism  $\phi$ of $\Omega_{m,n}$ given by 
%\[
%f^{\phi}(v,v')=[f(v^{\phi^{-1}},{v'}^{\phi^{-1}})]^p.
%\] 
If $D_{(a_0,\ldots, a_{e-1})}^{(k)}$ is the matrix of $f$ in the bases $\{s_0,\ldots,s_{m-1}\}$ and $\{s_0,\ldots,s_{n-1}\}$ for $V$ and $V'$ respectively, then $f^\phi$ is $D_{(a_0^p,\ldots, a_{e-1}^p)}^{(k)}$ by Remark \ref{rem_9}. This concludes the proof.
\end{proof}
\begin{remark}\label{rem_10}
{\em 
The isomorphism $\nu=\nu_{\{\sk_0,\ldots,\sk_{m-1};s_0'^{(k)},\ldots,s_{n-1}'^{(k)}\}}:\Omega_{m,n}\rightarrow \cD_{m,n}^{(k)}(\Fqd)$ is described as follows.  Let $V=\<u_0,\ldots, u_{m-1}\>_{\Fq}$ and $V'=\<u_0',\ldots, u_{n-1}'\>_{\Fq}$ and 
let $f\in\Omega_{m,n}$ with matrix $M_f$ over $\Fq$ in the  bases $\{u_0,\ldots, u_{m-1}\}$  and $\{u'_0,\ldots, u'_{n-1}\}$  of $V$ and $V'$. Since $\{u_0,u_1,\ldots, u_{m-1}\}$ is a basis for $V(m,q^d)$ and $\{u_0',u_1,\ldots, u_{n-1}'\}$ is
a basis for $V(n,q^d)$, we can extend the action of $f$ on $V\times V'$ to an action on $V(m,q^d)\times V(n,q^d)$ in the natural way. Let  $f(s_0^{(k)},s_j'^{(k)})=a_{j}\in\Fqd$, $j=0,\ldots, e-1$. By Remarks \ref{rem_6} and \ref{rem_7}, the matrix of the change of basis from $\{u_0,\dots, u_{r-1}\}$  to $\{s_0^{(k)},\ldots,s_0^{(k)}\}$ is $E_r^{-1}K_r$. Therefore, $\nu(f)=D_{\ba}^{(k)}=(E_m^{-1}K_m)^{t} M_f (E_n^{-1}K_n)$, with $\ba=(a_0,\ldots, a_{e-1})$.
Since change of bases in $V(m,q^d)\times V(n,q^d)$ preserves the rank of  bilinear forms, we have $\rank(f)=\rank(M_f)=\rank(D_\ba^{(k)})$.
}
\end{remark}

%
%

%\sout{The utility of using the cyclic models for the vector spaces $V$ and $V'$ and representing bilinear forms with $q^k$-circulant $m\times n$ matrices turns out to be particularly useful to calculating the automorphism groups of MRD codes, at least  when $m$ is a divisor of $n$, as we show in the following sections.}}

%
%
%%%%%%%%%%%%%%%%%%%%%%%%%%%%%%%%%%%%%%%%%%%%%%%%%%%%%%%%%%%%%%%%%%%%%%%%%
%%%%%%%%%%%%%%%%%%%%%%%%%%%%%%%%%%%%%%%%%%%%%%%%%%%%%%%%%%%%%%%%%%%%%%%%%
%%%%%%%%         From square MRD codes                   %%%%%%%%%%%%%%%%
%%%%%%%%        to rectangular MRD codes                 %%%%%%%%%%%%%%%%
%%%%%%%%%%%%%%%%%%%%%%%%%%%%%%%%%%%%%%%%%%%%%%%%%%%%%%%%%%%%%%%%%%%%%%%%%
%%%%%%%%%%%%%%%%%%%%%%%%%%%%%%%%%%%%%%%%%%%%%%%%%%%%%%%%%%%%%%%%%%%%%%%%%
%%%%%%%%%%%%%%%%%%%%%%%%%%%%%%%%%%%%%%%%%%%%%%%%%%%%%%%%%%%%%%%%%%%%%%%%%
%
%
\section{Puncturing generalized Gabidulin codes} \label{sec_3}

Let $\cX$ be a rank distance code in $M_{n,n}(\Fq)$ and $A$  any given $m\times n$ matrix of rank $m$, $m<n$. It is clear that the set  $A\cX=\{AM: M\in \cX \}$ is a rank distance code in $M_{m,n}(\Fq)$. We say that the code $A\cX$, which we will denote by $\cP_A(\cX)$, is  obtained by {\em puncturing  $\cX$ with $A$} and $\cP_A(\cX)$  is  known as  a {\em punctured code}.

\begin{theorem}[Sylvester's rank inequality] \cite[p.66]{gan}
Let $A$ be an $m \times n$ matrix and $M$ an $n\times n'$  matrix. Then
\[
\rank(AM)\ge \rank(A)+\rank(M)-n.
\]
\end{theorem}

\begin{theorem} (see also \cite[Corollary 35]{br})\label{th_6}
Let $\cX$ be an $(n,n,q;s)$-MRD code. Let $A$ be any $m \times n$ matrix  over $\Fq$ of rank $m$, with  $n-s  <  m \leq n$. Then the punctured code $\cP_A(\cX)$ is an $(m,n,q;s')$-MRD code, with $s'=s+m-n$.
\end{theorem}
\begin{proof}
We first show that the map $M\mapsto AM$ is injective. Assume $AM_1=AM_2$ for some distinct matrices $M_1, M_2 \in \cX$. Then
$A(M_1-M_2)=0$, giving $\dim(\ker A)\ge \rank\,(M_1-M_2)\ge s>0$, thus $\rank\, A=m-\dim(\ker A)<m$, a contradiction. Therefore, $|A \cX|=|\cX|=q^{n(n-s +1)}=q^{n(m-s'+1)}$.

By the  Sylvester's rank inequality, we have
\[
\rank(AM_1-AM_2)\ge \rank(A)+\rank(M_1-M_2)-n \ge m+s-n=s'>0.
\]
It follows that $A\cX$ is an $(m,n,q;s')$-MRD code.
% \qed
\end{proof}
\begin{remark}\label{rem_3}
{\em
Let $B$ be matrix in $M_{m,n}(\Fq)$  of rank $m$.  It is known that there exist $S\in\GL(m,q)$ and $T\in\GL(n,q)$ such that $B=SAT$  \cite[p.62]{gan}. Therefore
	\[
\cP_B(\cX)=\cP_{SAT}(\cX)=S\cP_{A}(T\cX),
	\]
	giving $\cP_B(\cX)$ is equivalent to  the punctured code  $\cP_A(T\cX)$. Note that $T\cX$ is equivalent to $\cX$.
	}
	\end{remark}
%
%\begin{remark}\label{rem_4}
%We notice that to apply Theorem \ref{th_6} the  MRD-code is not necessarily linear. Let  $\cX$ one of the non-linear MRD-code contructed in \cite{ds}. For any $A\in M_{m,n}(\Fq)$  of rank $m$, with  $1  <  m \leq n$, the punctured code $\cP_A(\cX)$ is a $(m,n,q;m-1)$-MRD code. {\bf We need to check if it is non-linear, at least for some particluar $A$}
%\end{remark}
%
We recall the construction of the generalized Gabidulin codes as given in \cite{gab}. For any positive integers $t,k$ with  $t\le n$ and $\gcd(k,n)=1$, set  $\cL_t^{(k)}(\Fqn)$ to be the set of all  $q^k$-polynomials over $\Fqn$ of $q^k$-degree at most $t-1$, i.e.
\[
\cL_t^{(k)}(\Fqn)=\{a_0+a_1x^{[k]}+\ldots+a_{t-1}x^{[k(t-1)]}:a_i \in \Fqn\}.
\] 
\\
We note that by reordering the powers of $x$ in any $f\in\cL_n^{(k)}(\Fqn)$ we actually find a $q$-polynomial. However,  to study the generalized Gabidulin codes in terms of $q^k$-polynomials we need to keep the original order for the powers in $f$.

Let $g_0,\ldots, g_{m-1}\in\Fqn$, $m\le n$, be linearly independent  over $\Fq$. Let $G^{[k]}$ be the matrix
\[
G^{[k]}=\left(\begin{array}{cccc}
g_0   & g_1   & \cdots & g_{m-1}  \\
g_0^{[k]} & g_1^{[k]} & \cdots & g_{m-1}^{[k]}\\
\cdots&\cdots & \cdots &\cdots\\
g_0^{[(t-1)k]} & g_1^{[(t-1)k]}  & \cdots & g_{m-1}^{[(t-1)k]}
\end{array}\right).
\]

We consider the matrix $G^{[k]}$ as a generator matrix of a subset $\tilde\cG_t^{(k)}$ of arrays over $\Fqn$, i.e.  $\tilde\cG_t^{(k)}  = \tilde\cG_{(g_0,\ldots, g_{m-1});t}^{(k)}=  \{(f(g_0),\ldots,f(g_{m-1})) :f\in \cL_t^{[k]}(\Fqn)\}$.

Let $V'=\Fqn=\<u_0,\ldots, u_{n-1}\>_{\Fq}$. The map
\[
\begin{array}{rccl}
\varepsilon=\varepsilon_{\{u_0,\ldots, u_{n-1}\}}:  & \Fqn & \longrightarrow & \Fq^n\\
                                         & \sum{x_iu_i}  & \mapsto & (x_0,\ldots, x_{n-1})^t
\end{array}
\]
maps the set $\tilde\cG_t^{(k)}$ to the matrix set
\[
\varepsilon(\tilde\cG_t^{(k)})=\{(M^{(0)}\,M^{(1)}\,\ldots\,M^{(m-1)}):f\in\cL^{(k)}_{t}(\Fqn)\}\subseteq M_{n,m}(\Fq),
\]
 where $M^{(i)}=\varepsilon(f(g_i))$. Since the rank is invariant under matrix transposition and in this paper  we consider matrix codes in  $M_{m,n}(\Fq)$   with $m\leq n$, we may take the matrix code $\cG_t^{(k)}$ obtained by taking the  transpose of the elements in  $\varepsilon(\tilde\cG_t^{(k)})$.  Therefore $\cG_t^{(k)}$  is a $(m,n,q;m-t+1)$-MRD code.  These MRD codes  are called {\em generalized Gabidulin  codes} \cite{kg}.

 By Proposition \ref{prop_3}, we may identify  the elements in $\End(V')$ with elements in $\cL^{(k)}_{n}(\Fqn)$  via the map $D_{(a_0,\ldots,a_{n-1})}\mapsto a_0+a_1x^{[k]}+\ldots+a_{n-1}x^{[k(n-1)]}$. Therefore,   $\cG_t^{(k)}$ consists of the matrices of the restriction over \ the subspace $V=\langle g_0,\ldots, g_{m-1}\rangle _{\Fq}$ of $V'=\Fqn$ of all the  endomorhisms of $V'$. These matrices act on the set $\Fq^m$  of all row vectors as $v\mapsto vM$.  As we are working in the framework of bilinear forms, we consider  any matrix in $\cG_t^{(k)}$ as a matrix of the restriction on $V\times V'$ of the bilinear form acting on $V'\times V'$ whose $n\times n$ matrix  is the  matrix of an element in $\cL_{n}^{(k)}(\Fqn)$. By Proposition \ref{prop_8} the elements in $\cG_t^{(k)}$ can be represented by $q^k$-circulant matrices over $\Fqd$, where $d=\gcd(m,n)$.

{he following result seems to be known, but we include a proof for the sake of completness.
\begin{theorem}\label{prop_9}
Let $\cG$ be any  generalized Gabidulin $(n,n,q;n-t+1)$-code and let $A$ be any given $m\times n $ matrix over $\Fq$ of rank $m$, with  $t  <  m \leq n$. Then the punctured code $\cP_A(\cG)$ is a generalized  Gabidulin $(m,n,q;m-t+1)$-code. Conversely, every generalized  Gabidulin  $(m,n,q;m-t+1)$-code, with $1\le t\le m$, is obtained by puncturing a generalized Gabidulin $(n,n,q;n-t+1)$-code.
\end{theorem}
\begin{proof}
Let $V'=\Fqn=\<u_0,\ldots, u_{n-1}\>_{\Fq}$. By the argument above, $\cG$ is  considered as the set of all bilinear forms acting on $V'\times V'$ whose matrix corresponds to a  $q^k$-polynomial in $\cL_{t}^{(k)}(\Fqn)=\{a_0+a_1x^{[k]}+\ldots+a_{t-1}x^{[k(t-1)]}:a_i \in \Fqn\}$.

The given matrix $A=(a_{ij})$ corresponds to the linear transformation
\[
\begin{array}{rccl}
\tau \colon & u_i & \mapsto & \sum_{j=0}^{n-1}{a_{ij}u_j},
\end{array}\ \ \ i=0,\ldots, m-1.
\]
As $\rank\, A=m$, the subspace $V=\langle \tau(u_0),\ldots,\tau(u_{m-1})\rangle $ is an $m$-dimensional subspace of $V'$. It follows that    $\cP_A(\cG)$ consists of the matrices of the bilinear forms   on $V\times V'$ in the bases $\{g_i=\tau(u_i): i=0,\ldots,m-1\}$ and $\{u_0,\ldots, u_{n-1}\}$ of $V$ and $V'$, respectively. Therefore $\cP_A(\cG)$ is the  generalized Gabidulin code $\cG_{(g_0,\ldots,g_{m-1}),t}^{(k)}$. By Theorem \ref{th_6} $\cG_{(g_0,\ldots,g_{m-1}),t}^{(k)}$  is an  $(m,n,q;s+m-n)$-MRD code.

For the converse, let $\cG_t^{(k)}=\cG_{(g_0,\ldots,g_{m-1});t}^{(k)}$ be a generalized Gabidulin  code.  Set $V=\langle g_0,\ldots, g_{m-1}\rangle _{\Fq}$ and
extend $g_0,\ldots, g_{m-1}$ with $g_m,\ldots, g_{n-1}$ to form a basis of  $V'=\Fqn=V(n,q)$. Then, elements in  $\cG_t^{(k)}$  are the restriction on
$V\times V'$ of the bilinear forms acting on $V'\times V'$ whose matrix  is the  matrix of elements in $\cL_{t}^{(k)}(\Fqn)$ in  the basis $g_0,\ldots, g_{n-1}$. The set $\overline \cG_t^{(k)}=\cG_{(g_0,\ldots, g_{n-1});t}^{(k)}$  of such bilinear forms is a  generalized Gabidulin  $(n,n,q;n-t+1)$-code.
In addition, matrices in $\cG_t^{(k)}$ are obtained from the matrices of $\overline\cG_t^{(k)}$ by deleting the last $n-m$ rows, i.e. $\cG_t^{(k)}=A\overline \cG_t^{(k)}$ with $A=(I_m|O_{n-m})$. Therefore, the generalized Gabidulin code $\cG_t^{(k)}$ is obtained by puncturing  the generalized Gabidulin code $\overline \cG_t^{(k)}$ with $A$.
\end{proof}
From the proof of the previous result we get the following description for the generalized Gabidulin  codes.

\begin{corollary}
Let $g_0,\ldots, g_{m-1}\in\Fqn$, $m\le n$, be linearly independent  over $\Fq$. Then the 
 generalized Gabidulin  code $\cG_{(g_0,\ldots,g_{m-1});t}^{(k)}$ is the set of all the bilinear forms acting on  $V\times V'$ with $V=\<g_0,\ldots,g_{m-1}\>_{\Fq}$ and $V'=\Fqn$.
\end{corollary}

\begin{remark}
{\em  Let $e=\gcd(m,n)$ and $d=\lcm(m,n)$. From the arguments contained in Section  \ref{sec_2} there exists a (Singer) basis of $V=\<g_0,\ldots,g_{m-1}\>_{\Fq}$ and a (Singer)  basis of $V'=\Fqn$ such that the elements in $\cG_{(g_0,\ldots,g_{m-1});t}^{(k)}$ may be represented as $m\times n$ $q^k$-circulant matrices over $\Fqd$.
}
\end{remark}

\begin{remark}
{\em
By the isomorphism $\Omega_{m,n}\simeq\cD_{m,n}^{(1)}(\Fqd)$ stated by Proposition \ref{prop_8}, the Gabidulin code $\cG_{(g_0,\ldots,g_{m-1}),t}^{(1)}$ is actually the Delsarte code defined by (6.1) in \cite{del} with $V=\langle g_0,\ldots, g_{m-1}\rangle _{\Fq}$.
}
\end{remark}

In the rest of the paper, $m$  will be a divisor of $n$. We set $r=n/m$. Let $V=\langle  u_0,\ldots, u_{m-1}\rangle $ and $V'=\langle u'_0,\ldots, u'_{n-1}\rangle $ be  two vector spaces over $\Fq$ of dimension $m$ and $n$, respectively. If $m=n$ we take $V'=V=\langle  u_0,\ldots, u_{m-1}\rangle$.

In the light of the isomorphism  $\nu_{\{u_0,\ldots, u_{m-1};u'_0,\ldots, u'_{n-1}\}}$ described by (\ref{eq_4}), every bilinear form acting on $V\times V'$ may be identified  with an  $m\times n$ matrix over $\Fq$. In other words, if we assume $V$ is an $m$-dimensional subspace of $V'$ after a vector-space isomorphism, then the bilinear forms in $\Omega_{m,n}$ are the restrictions on $V\times V'$ of the bilinear form in $\Omega_{n,n}$. Thus,   $\Omega_{m,n}$ is the puncturing of $\Omega_{n,n}$ by a suitable $m\times n$ matrix of rank $m$.

 In this paper we work with cyclic models for  vector spaces over $\Fq$.  Let $\{s_0,\ldots, s_{n-1}\}$ be a Singer basis for $V'$. We  note that not all $m$-dimensional subspaces of $V'$ can be represented with a cyclic model over $\Fqm$. Therefore, we need to choose suitable vectors $\{s'_0,\ldots, s'_{m-1}\}$ in $V'$ such that the projection of the vectors in the cyclic model for $V'$  on the subspace spanned by $\{s'_0,\ldots, s'_{m-1}\}$ gives a cyclic model for $V=V(m,q)$. This is what we do in the rest of this section.

Let $\sigma'$ be a Singer cycle of $V'$ with associated primitive element $w'$. Let $\{s_0',\ldots,s_{n-1}'\}$ be  the Singer basis for $V'$ defined by $\sigma'$. Note that $s_i'\in V(n,q^n)$, for $i=0\ldots, n-1$.
 Set $s_i=\sum_{j=0}^{r-1} s_{i+jm}'$ for $i=0,1,\ldots, m-1$, $\sigma=\sigma'^{(q^{n}-1)/(q^m-1)}$ and  $w=w'^{(q^{n}-1)/(q^m-1)}$. Then $\sigma$ has order $q^m-1$ and $\omega$ is a primitive element of $\Fqm$ over $\Fq$. It is easily seen that  $s_{i}$ is an eigenvector for $\sigma$ with eigenvalue $w^{q^i}$, for $i=0,\ldots,m-1$.  Since $m$ divides $n$, the $\Fqm$-span $V(m,q^m)$ of $\{s_0,\ldots, s_{m-1}\}$ is contained in $V(n,q^n)$.  Let $\xi$ be the semilinear transformation on $V(m,q^m)$ whose linear part is defined by $\ell_{\xi}(s_i)=s_{i+1}$,  where the indices are considered modulo $m$,
and whose companion  automorphism is $\delta:x\in\Fqm\mapsto x^q\in\Fqm$. Since the subset  $\{xs_0+\ldots +x^{q^{m-1}}s_{m-1}:x \in \Fqm\}$ of $V(m,q^m)$ is fixed pointwise by $\xi$, it is a cyclic model for $V$. By Proposition \ref{prop_8}, every bilinear form  on $V\times V'$  can be represented by $m\times n$ $q$-circulant matrices over $\Fqn$.
 
\begin{lemma}\label{lem_5}
Let $k$ be a positive integer such that $\gcd(k,n)=1$. Then $s_i^{(k)}=\sum_{j=0}^{r-1} s_{i+jm}'^{(k)}$, for $i=0,1,\ldots, m-1$.
\end{lemma}
\begin{proof}
For  $i=0,1,\ldots, m-1$ we have $s_i^{(k)}=s_{ki\bmod m}=\sum_{j=0}^{r-1} s_{(ki\bmod m)+jm}'$. On the other hand, $s_{i+jm}'^{(k)}=s_{k(i+jm)\bmod n}$. Therefore we need to prove
\[
\{(ki \bmod m)+jm:j=0,\ldots, r-1\}=\{k(i+lm)\bmod n:l=0,\ldots, r-1\}.
\]
Let $ki=tm+s$ with $0\leq s \leq m-1$. For each $j\in \{0,1,\ldots,r-1\}$ we need to find $l\in \{0,1,\ldots,r-1\}$ such that $s+jm \equiv (tm+s+klm) \bmod n$.
 This is equivalent to 
\begin{equation}
\label{cc}
j \equiv (t+ kl)\bmod r,
\end{equation}
as  $r=n/m$. 
Since $\gcd(k,n)=1$, we also have $\gcd(k,r)=1$. Let $k^{-1}$ denote the inverse of $k$ modulo $r$.
With $l=k^{-1}(j-t)\bmod r$ equation (\ref{cc}) is satisfied and
\[
\{k^{-1}(j-t)\bmod r \colon j \in \{0,1,\ldots,r-1\}\}=\{0,1,\ldots,r-1\}.
\]
Hence the assertion is proved.
% \qed
\end{proof}

The above Lemma implies the following result.
\begin{proposition}\label{prop_10}
Let $k$ be a positive integer such that $\gcd(k,n)=1$ and $m$ a divisor of $n$. Let $\{s_0',\ldots,s_{n-1}'\}$ be a Singer basis for $V'$.  Set $s_i^{(k)}=\sum_{j=0}^{r-1} s_{i+jm}'^{(k)}$, for $i=0,1,\ldots, m-1$.
Then the  $\Fq$-subspace $\left\{\sum_{i=0}^{m-1}{x^{[ki]}s_i^{(k)}} \colon x \in \Fqm\right\}$ of $V(n,q^n)$ is  a  $k$-cyclic model  for $V$.
\end{proposition}
\begin{proof}
By Remark \ref{rem_7} the $k$-cyclic model  for $V$ is obtained from the cyclic model $V=\{\sum_{i=0}^{m-1}{x^{q^i}s_i} \colon x \in \Fqm\}$  by applying the change of basis, say $\kappa$, from $\{s_0,\ldots,s_{m-1}\}$ with $s_i=\sum_{j=0}^{r-1} s_{i+jm}'$ for $i=0,1,\ldots, m-1$, to $\{s_0^{(k)},\ldots,s_{m-1}^{(k)}\}$. We recall that $\kappa$ is represented by the permutation matrix $K_{m}^{-1}$. By Lemma \ref{lem_5}, $s_i^{(k)}=\sum_{j=0}^{r-1} s_{i+jm}'^{(k)}$, for $i=0,1,\ldots, m-1$. This implies that   the image under $\kappa$ of the  cyclic model  $\{\sum_{i=0}^{m-1}{x^{q^i}s_i} \colon x \in \Fqm\}$ is $\left\{\sum_{i=0}^{m-1}{x^{[ki]}s_i^{(k)}} \colon x \in \Fqm\right\}$.
\end{proof}

Let $f$ be any given bilinear form acting on the $k$-th cyclic model of $V'$ with $q^k$-circulant matrix  $D_\ba^{(k)}$ in the basis $\{s_0'^{(k)},\ldots, s_{n-1}'^{(k)}\}$. As the matrix of the coordinates of the vectors $s_0^{(k)},\ldots, s_{m-1}^{(k)}$ in this basis is the $1\times r$ block matrix $A=(I_m \mid I_m \mid \ldots \mid I_m)$, the restriction $f|_{V\times V'}$ of $f$ on $ V\times V'$ has matrix $\overline D=AD_\ba^{(k)}$ in the bases $\{s_0^{(k)},\ldots, s_{m-1}^{(k)}\}$ and $\{s_0'^{(k)},\ldots, s_{n-1}'^{(k)}\}$.

To make notation easier, we index  the rows and columns of an $m\times n$ matrix $M$ by  elements in $\{0,\ldots,m-1\}$ and $\{0,\ldots, n-1\}$. Further, $M_{(i)}$ and $M^{(j)}$ will denote the $i$-th row and $j$-th column of $M$, respectively.

For $i=0,\ldots,m-1$, we have $A_{(i)}=(0,\ldots,0,1,0,\ldots,0,1,0\ldots,0,1,0\ldots)$,
with 1 at position $i+km$, for $k=0,\ldots, r-1$, and  0 elsewhere.
Let $D_\ba^{(k)}$ be generated by the array   $\ba=(a_0,\ldots,a_{n-1})$.  Then,
$[D_\ba^{(k)}]^{(j)}=(a_j,a_{j-1}^{[k]},\ldots,a_{j+1}^{[k(n-1)]})^t$, for $j=0,\ldots,n-1$.
Therefore, the $(i,j)$-entry of $\overline{D}$   is
\[
a_{j-i}^{[ki]}+a_{j-(i+m)}^{[k(i+m)]}+\ldots+a_{j-(i+(r-1)m)}^{[k(i+(r-1)m)]}=\sum_{h=0}^{r-1}a_{j-(i+hm)}^{[k(i+hm)]},
\]
where indices are taken modulo $n$. It turns out that  $\overline D$ is the $m\times n$ $q^k$-circulant matrix   over $\Fqn$  generated  by the array $(\sum_{h=0}^{r-1}a_{-hm}^{[khm]},\sum_{h=0}^{r-1}a_{1-hm}^{[khm]},\ldots, \sum_{h=0}^{r-1}a_{m-1-hm}^{[khm]})$.
In particular, if $\ba=(a_0,\ldots,a_{t-1},0,\ldots,0)$ for some $t\in\{1,\ldots, m\}$, then $\overline D$ is   generated  by the $m$-array  $(a_0,\ldots,a_{t-1},0,\ldots,0)$ giving  $\overline D$ to be the matrix of a bilinear form in the  rank distance code
\begin{equation}\label{eq_25}
\Phi_{m,n,t}^{(k)}=\bigoplus_{j=0}^{t-1}{\Omega_{j}^{(k)}},
\end{equation}
where the $\Fq$-subspaces $\Omega_{j}^{(k)}$ are given in (\ref{eq_28}). Note that the dimension of $\Phi_{m,n,t}$ over  $\Fq$ is $nt$.
The above arguments together with Theorem \ref{prop_9}  prove the following result.

\begin{theorem}
Let   $m>1$ be  any divisor of $n$.  For any  $t\in\{1,\ldots, m\}$,  the  subset $\Phi_{m,n,t}^{(k)}$ of  $\Omega_{m,n}$   is a generalized Gabidulin $(m,n,q;m-t+1)$-code.
\end{theorem}

We now describe the generalized twisted Gabidulin codes as provided by Sheekey  in the recent paper \cite{she} by using the framework of  $q^k$-polynomials over $\Fqn$.   In \cite{ltz}  the equivalence between different  generalized twisted Gabidulin codes  was addressed.

Let $\N_{q^n/q}$ denote the norm map from $\Fqn$ onto $\Fq$:
\[
\N_{q^n/q} \colon y\in\Fqn \mapsto \N_{q^n/q}(y)=\prod_{j=0}^{n-1}{y^{q^j}}\in \Fq.
\]

\begin{theorem}\cite{ltz,she}\label{th_5}
 For any   $t\in\{1,\ldots,n-1\}$ and  $\mu\in \Fqn$ with $\N_{q^n/q}(\mu)\neq (-1)^{nt}$, define the subset
\[
\Gamma_{n,n,t,\mu,s}^{(k)}=\{f_{a,0}^{(k)}+f_{\mu a^{q^{sk}}, t}^{(k)}:a\in \Fqn\}
\]
of $\Omega_{n,n}$ and put 
\[
\cH_{n,n,t,\mu,s}^{(k)}=\Gamma_{n,n,t,\mu,s}^{(k)}\oplus \Omega_{1}^{(k)}\oplus\ldots\oplus\Omega_{t-1}^{(k)}.
\]
Then  $\cH_{n,n,t,\mu,s}^{(k)}$  is an $(n,n,q;n-t+1)$-MRD code which is not equivalent to $\Phi_{n,n,t}^{(k)}$ if $t \neq 1,n-1$. 
\end{theorem}

Let  $A$ be the $1\times r$ block matrix $(I_m \mid I_m \mid \ldots \mid I_m)$. By  Theorem \ref{th_6}, for any  $t\in\{1,\ldots,m-1\}$, the punctured code  $\cP_A(\cH_{n,n,t,\mu,s}^{(k)})$ is an MRD code with the same parameters as the code $\Phi_{m,n,t}^{(k)}$  defined by (\ref{eq_25}). We  denote such $(m,n,q;m-t+1)$-MRD code  by  $\cH_{m,n,t,\mu,s}^{(k)}$. By using the decomposition (\ref{eq_7}) of $\Omega_{m,n}^{(k)}$ with $e=m$, we get
\begin{equation}\label{eq_29}
\cH_{m,n,t,\mu,s}^{(k)}=\Gamma_{m,n,t,\mu,s}^{(k)} \oplus \Omega_{1}^{(k)}\oplus\ldots\oplus\Omega_{t-1}^{(k)}
\end{equation}
where $\Gamma_{m,n,t,\mu,s}^{(k)}=\{f_{a,0}^{(k)}+f_{\mu a^{q^{sk}}, t}^{(k)}:a\in \Fqn\}$ and $f_{a,j}^{(k)}\in\Omega_{j}^{(k)}$ is defined by (\ref{eq_6}). It turns out that $\cH_{m,n,t,\mu,s}^{(k)}$ is an linear MRD code of dimension $nt$.
\begin{remark}
{\em
It is easy to see that  the MRD code $\cH_{m,n,t,\mu,s}^{(k)}$ has the same parameters as  the $(m,rm/2,q;m-1)$-MRD codes provided in \cite{bmpz}  only for $t=2$ and $n=m$.
}
\end{remark}

%
%%%%%%%%%%%%%%%%%%%%%%%%%%%%%%%%%%%%%%%%%%%%%%%%%%%%%%%%%%%%%%%%
%%%%%%%%%%%%%%%%%%%%%%%%%%%%%%%%%%%%%%%%%%%%%%%%%%%%%%%%%%%%%%%%
%%%%%%%% The automorphism group of Delsarte's  MRD codes  %%%%%%
%%%%%%%%%%%%%%%%%%%%%%%%%%%%%%%%%%%%%%%%%%%%%%%%%%%%%%%%%%%%%%%%
%%%%%%%%%%%%%%%%%%%%%%%%%%%%%%%%%%%%%%%%%%%%%%%%%%%%%%%%%%%%%%%%
%
%
\section{The automorphism group of some punctured  generalized Gabidulin codes} \label{sec_4}

Very recently, Liebhold and  Nebe calculated the group of all linear automorphisms of any generalized Gabidulin code.  Here we give the finite fields version of Theorem 4.9 in \cite{lnebe}.
\begin{theorem}\cite{lnebe}\label{th_8}
Let $\cG_t^{(k)}=\cG_{(g_0,\ldots, g_{m-1}),t}^{(k)}$ be a  generalized Gabidulin code. Let $\Fq\le \mathbb{F}_{q^{m'}} \le \Fqn$ be the maximal subfield of $\Fqn$ such that $\langle g_0,\ldots, g_{m-1}\rangle _{\mathbb{F}_q}$ is an $\mathbb{F}_{q^{m'}}$-subspace. Let  $\mathbb{F}_{q^d}$ be the minimal subfield of $\Fqn$ such that $\langle g_0,\ldots, g_{m-1}\rangle _{\Fq}$ is contained in a one-dimensional $\mathbb{F}_{q^d}$-subspace. Then there is a subgroup $G$ of $\Gal(\mathbb{F}_{q^d}/\Fq)$ such that the group of all linear automorphisms of $\cG_t^{(k)}$ is isomorphic to 
\[
(\mathbb{F}_{q^{m'}}^\times \times \GL(n/d,q^d)  )\rtimes G.
\]
\end{theorem}
\begin{remark}
{\em
The generalized Gabidulin code $\Phi_{m,n,t}^{(k)}$ defined by (\ref{eq_25}) corresponds to the case $\langle g_0,\ldots, g_{m-1}\rangle _{\Fq}=\Fqm$ in  the previous theorem.
}
\end{remark}

In spite of Theorem \ref{th_8}, we believe that  it is useful to have an explicit description of the full  automorphism group of an MRD code to compare MRD codes among each other; see \cite{she,tr}.

 Let $S'=\langle \sigma'\rangle $ be a Singer cyclic group of   $\GL(V')$  with associated semilinear transformation $\xi'$  as described in Section \ref{sec_2}. In the basis $\{s_0'^{(k)},\ldots,s_{n-1}'^{(k)}\}$, the matrix of $\sigma'$ is the diagonal matrix $\diag(w,w^{[k]},\ldots,w^{[k(n-1)]})$. Therefore, $(\sigma'^{i},\sigma'^{j})$ acts on $\Omega_{n,n}$ by mapping the bilinear form with $q^k$-circulant matrix $D_{(a_0,\ldots, a_{n-1})}^{(k)}$ to  the bilinear form with matrix $D_{(w^ia_0w^j,\ldots, w^ia_{n-1}w^{j[k(n-1)]})}^{(k)}$.
The matrix of the linear part $\ell_{\xi'}$ of $\xi'$  is the permutation matrix $D_{(0,\ldots,0,1)}^{(k)}$. Then $(\ell_{\xi'},\ell_{\xi'})$ acts on $\Omega_{n,n}$ by mapping the bilinear form with $q^k$-circulant $D_{(a_0,\ldots, a_{n-1})}^{(k)}$ to  the bilinear form with  matrix $D_{(a_0^{[k]},\ldots,a_{n-1}^{[k]})}^{(k)}$. Set $\bar \ell=(\ell_{\xi'},\ell_{\xi'})$ and $\bar C$ be the cyclic subgroup of $\Aut(\Omega_{n,n})$ generated by $\bar\ell$. It turns out that  any element in $(S'\times S')\rtimes \bar C\rtimes \Aut(\Fq)$  fixes every component $\Omega_{j}^{(k)}$ of $\Omega_{n,n}$.

In the paper \cite{she}, Sheekey gave a complete description of the automorphism group of the MRD codes  $\Phi_{n,n,t}^{(k)}$  and $\cH_{n,n,t,\mu,s}^{(k)}$, for any $t\in\{1,\ldots,n-2\}$.

\begin{theorem}\cite{she}\label{th_9} Let $q=p^h$, $p$ a prime. 
\begin{itemize}
\item[i)]For any given $t\in\{1,\ldots,n-2\}$, the automorphism group of  $\Phi_{n,n,t}^{(k)}$ is the semidirect product $(S'\times S')\rtimes \bar C\rtimes \Aut(\Fq)$.

\item[ii)] For any given $t\in\{1,\ldots,n-2\}$, the automorphism group of $\cH_{n,n,t,\mu,s}^{(k)}$  is the subgroup of $(S'\times S')\rtimes \bar C \rtimes \Aut(\Fq)$ whose elements correspond to the triples $((D_{\ba},D_{\bb});\bar\ell^i; p^e)$, with $\ba=(a,0,\ldots,0)$, $\bb=(b,0,\ldots,0)$  and $a^{q^s-1}b^{q^s-q^t}=\mu^{p^eq^i-1}$.
\end{itemize}
\end{theorem}
We now give more informations on the automorphism group of the MRD code $\cH_{n,n,t,\mu,s}^{(k)}$.

\begin{corollary}
$|\Aut(\cH_{n,n,t,\mu,s}^{(k)})|=l(q^n-1)(q^{gcd(n,s,t)}-1)$, for some divisor $l$ of $nh$.  If $\mathbb F_{p^d}$ is the smallest subfield of $\Fqn$ which contains $\mu$, then $\frac{nh}{d}\, |\, l$.
In particular, if $\mu\in \Fp^*$, then  $l=nh$.

\end{corollary}
\begin{proof}
Let $G_{n,k}=\{x^{q^k-1} \colon x\in \Fqn^*\}$.
Note that $G_{n,k}$ is a subgroup of the multiplicative group of $\Fqn^*$.
To calculate $|G_{n,k}|$ it is enough to observe that $x^{q^k-1}=y^{q^k-1}$ for some $x,y\in \Fqn^*$ if and only if
$x/y \in \Fqk^*$, and hence $x/y \in \mathbb F_{q^{\gcd(n,k)}}^*$. It follows that
$|G_{n,k}|=\frac{q^n-1}{q^{\gcd(n,k)}-1}$.
 Also, the size of the subgroup $G_{n,k} \cap G_{n,j}$ is
\[c_{n,k,j}=\gcd\left(\frac{q^n-1}{q^{\gcd(n,k)}-1},\frac{q^n-1}{q^{\gcd(n,j)}-1}\right).\]
Note that for each $c\in \Fqn$ we have $|G_{n,k} \cap cG_{n,j}|\in \{0,c_{n,k,j}\}$.

Let $d_{n,k,j}$ denote the number of pairs $(x,y)\in \Fqn^*\times \Fqn^*$ such that
\[
x^{q^k-1}y^{q^j-1}=1.
\]

If $x^{q^k-1}y^{q^j-1}=1$, then $x^{q^k-1}=(1/y)^{q^j-1}$ and hence $x^{q^k-1} \in G_{n,k} \cap G_{n,j}$.
It follows that we can choose $x_0=x^{q^k-1}$ in $c_{n,k,j}$ different ways, and it uniquely defines $y_0=y^{q^j-1}$.
We have
\[
|\{x \colon x^{q^k-1}=x_0\}|=|\{x \colon x^{q^k-1}=1\}|=q^{\gcd(n,k)}-1
\]
and
\[
|\{y \colon y^{q^j-1}=y_0\}|=|\{y \colon  y^{q^j-1}=1\}|=q^{\gcd(n,j)}-1,
\]
thus
\[
d_{n,k,j}=(q^n-1)(q^{\gcd(n,k,j)}-1).
\]

If $x^{q^k-1}y^{q^j-1}=c$, then $x^{q^k-1}=c (1/y)^{q^j-1}$, thus $x^{q^k-1} \in G_{n,k} \cap cG_{n,j}$.
From the above arguments, we get that for each $c\in \Fqn^*$, the number of pairs $(x,y)\in \Fqn^*\times \Fqn^*$ such that $x^{q^k-1}y^{q^j-1}=c$ is either 0, or $d_{n,k,j}$.

Now, let $\mu$ be a given element in  $\Fqn^*$, $q=p^h$, and let $\cH$ denote the set of integers $r$, such that
\[
x^{q^k-1}y^{q^j-1}=\mu^{p^r-1}
\]
has a solution in $\Fqn^*\times\Fqn^*$. By the above arguments, we have $0\in \cH$. If $x_0^{q^k-1}y_0^{q^j-1}=\mu^{p^r-1}$ and $x_1^{q^k-1}y_1^{q^j-1}=\mu^{p^s-1}$, then
\[
(x_0^{p^s})^{q^k-1}(y_0^{p^s})^{q^j-1}=\mu^{p^{r+s}-p^s}\]
 and
\[
(x_0^{p^s}x_1)^{q^k-1}(y_0^{p^s}y_1)^{q^j-1}=\mu^{p^{r+s}-1},
\] thus
$r,s \in \cH$ yields $r+s \in \cH$ giving  $\cH$ is an additive subgroup in $\mathbb{Z}_{nh}$. Therefore, $l=|\cH|$ divides $nh$.

%Assume $s\ge t$.
By the above arguments, the number of triples $(a,b,i)\in \Fqn^*\times \Fqn^*\times \mathbb{Z}_{nh}$ such that $a^{q^s-1}b^{q^s-q^t}=\mu^{p^{i}-1}$
is $l(q^n-1)(q^{gcd(n,s,t)}-1)$, for some divisor $l$ of $nh$.
If $\mathbb F_{p^d}$ is the smallest subfield of $\Fqn$ which contains $\mu$, then $\cH$ contains the additive subgroup of $\mathbb{Z}_{nh}$ generated by $d$, giving $
\frac{nh}{d}\, |\, l$.

If $\mu\in \Fp^*$, then $\mu^{p^i-1}=1$ for all $i\in \bZ_{nh}$.
% \qed
\end{proof}
Let $m>1$ be any divisor of $n$ and $k$ any positive integer such that $\gcd(k,n)=1$. Let  $S=\langle \sigma\rangle $ and $S'=\langle  \sigma' \rangle $ be  Singer cyclic groups of $\GL(V)$ and $\GL(V')$, respectively, and $\{s_0,\ldots, s_{m-1}\}$ and $\{s_0',\ldots, s_{n-1}'\}$ the Singer bases defined by $\sigma$ and $\sigma'$. 

 Set $\ell=(\ell_\xi,\ell_{\xi'})$ and let $C$ be the cyclic subgroup of $\Aut(\Omega_{m,n})$ generated by $\ell$.  Then $C\simeq\Aut(\Fqm/\Fq)$.

\begin{remark}\label{rem_5}
{\em 
To make notation easier, in all the arguments used in the actual  and in the next section we assume $k=1$. We put $\Omega_{j}=\Omega_{j}^{(1)}$ and  $\Phi_{m,n,t-1}=\Phi_{m,n,t-1}^{(1)}$.
 By Lemma \ref{lem_5},  the same  arguments work perfectly well for any $k$ with $\gcd(k,n)=1$ if the cyclic model of vector  spaces and  $q$-circulant matrices involved are replaced by the $k$-th cyclic model and $q^{k}$-circulant matrices. The details are left to the reader.
 }
\end{remark}

Both the Singer cyclic groups $S=\langle \sigma\rangle $ and $S'=\langle  \sigma' \rangle $, as well as the cyclic group $C$, fix every component $\Omega_{j}$ of $\Omega_{m,n}$ giving that every element in $(S\times S')\rtimes C$ is an automorphism of  $\Phi_{m,n,t}$, for any $t\in\{1,\ldots,m\}$.

\begin{theorem}\label{th_2}
%Let $V=V(m,q)$ and $V'=V(n,q)$ with  $m>1$  any divisor of $n$. Set $r=n/m$. 
Let $\overline T'$ be the subset of $\End(V')$ whose elements correspond to $q$-circulant matrices defined by an array of type
\[
(c_0,\underbrace{0,\ldots,0}_{m-1\mathrm{ \ times}},c_{1},\underbrace{0,\ldots,0}_{m-1\mathrm{ \ times}}, c_{r-1},\underbrace{0,\ldots,0}_{m-1\mathrm{ \ times}})
\]
 over $\Fqn$, and set $T'=\overline T'\cap\GL(V')$. Then, for any given $t\in\{1,\ldots,m-1\}$, the automorphism group of $\Phi_{m,n,t}$ is  the semidirect product $(S\times T')\rtimes C\rtimes \Aut(\Fq)$.
 \end{theorem}
\begin{proof}

Straightforward calculations show that the given group is a subgroup of $\Aut(\Phi_{m,n,t})$.

 Let $\varphi=(A,B;\theta)\in\Aut(\Phi_{m,n,t})$, with $A\in \GL(V)$ and $B\in \GL(V')$. As $\Phi_{m,n,t}$ is fixed by the semilinear automorphism $\phi$ defined by the Frobenius map $x\mapsto x^p$, we may assume $\theta=\1$.  We identify the elements $A$ and $B$ with their Dickson matrices in $\cB_{m}(\Fqm)$ and $\cB_{n}(\Fqn)$, respectively. To suit our present needs, we set $A^t=D_{(a_0,a_1,\ldots,a_{m-1})}$ and $B=D_{(b_0,b_1,\ldots,b_{n-1})}^t$.

 Let $f$ be any given element in $\Phi_{m,n,t}$ with $q$-circulant matrix $D_{\ba}$. Then, the Dickson  matrix of $f^\varphi$ is defined by the $m$-tuple formed by  the first $m$ entries of  $(A^t D_{\ba}B)_{(0)}$.

The $l$-th entry, with $0\le l\le m-1$, in $(A^t D_{\ba}B)_{(0)}$ is given by the inner product $(A^t D_{\ba})_{(0)}\cdot B^{(l)}$, with $B^{(l)}=(b_{-l}^{q^l},b_{-l+1}^{q^l},\ldots,b_{-l+n-1}^{q^{l}})$ where subscripts are taken modulo $n$.

We recall  that  $M^{(i)}$ and $M_{(j)}$ denotes the $i$-th column and the $j$-th row of  $M$, respectively.

Let $f=f_{\alpha,j}$ with $\alpha\in \Fqn$ and $0\le j\le m-1$. Then $D_{\ba}$ has only $n$ non-zero entries, one in each column. More precisely, the non-zero entry of the $h$-th column of $D_{\ba}$ is $\alpha^{q^{h-j}}$ at position $(h-j) \bmod m$. It follows that the $h$-th entry of
\[
(A^t D_{\ba})_{(0)}=(A^t_{(0)}D_{\ba}^{(0)},A^t_{(0)}D_{\ba}^{(1)},\ldots,A^t_{(0)}D_{\ba}^{(n-1)})
\]
is $a_{h-j}\alpha^{q^{h-j}}$, where the subscript $h-j$ is taken modulo $m$. Hence, the $l$-th entry of $(A^tD_{\ba}B)_{(0)}$ is
\begin{equation}
\label{eq_19}
\sum_{h=0}^{n-1}a_{h-j}\alpha^{q^{h-j}}b_{h-l}^{q^l}=
\sum_{h=0}^{m-1}a_{h-j}\sum_{k=0}^{r-1}{\alpha^{q^{h-j+km}}b_{h-l+km}^{q^l}}.
\end{equation}

Since we are assuming that $\varphi$ fixes $\Phi_{m,n,t}=\bigoplus_{j=0}^{t-1}{\Omega_{j}}$,  we must have (by putting $h-j=i$)
\begin{equation}\label{eq_9}
\sum_{i=0}^{m-1}a_{i}\sum_{k=0}^{r-1}{\alpha^{q^{i+km}}b_{i+j-l+km}^{q^l}}=0
\end{equation}
for $0\le j\le t-1$,  $t\le l\le m-1$.

%We explicitly note that, for any $l\in\{m-t,\ldots,m-1\}$,
%\[
%\begin{array}{l}
%\sum_{i=0}^{m-1}{(a_i\alpha^{q^i}b_{2m-l+j+i}^{q^l}+a_i\alpha^{q^{m+i}}b_{m-l+j+i}^{q^l})^{q^m}}=\\[.1in]
%\sum_{i=0}^{m-1}{(a_i\alpha^{q^{m+i}}b_{2m-l+j+i}^{q^{l+m}}+a_i\alpha^{q^{i}}b_{m-l+j+i}^{q^{l+m}})}
%\end{array}
%\]
%is the $(l+m)$-th entry in $(A^tD_\a B)_{(0)}$. Thus, we may restrict the range of variability of $l$ to $\{m-t,\ldots,m-1\}$.

Since equation (\ref{eq_9}) holds for all $\alpha\in\Fqn$, we get
%\begin{equation}\label{eq_10}
\[
a_ib_{i+j-l}= a_ib_{i+j-l+m}= \cdots= a_ib_{i+j-l+(r-1)m}=0,
\]
%\end{equation}
for  $0\le i \le m-1$ and  $t\le l\le m-1$.

As $A\in\cB_m(\Fqm)$,  $(a_0,a_1,\ldots,a_{m-1})\neq(0,\ldots,0)$.  By applying a suitable element in $C$, we may assume $a_0\neq 0$. Therefore, we get
\[
b_{j-l}= b_{j-l+m}=\ldots=b_{j-l+(r-1)m}=0,
\]
for   $0\le j\le t-1$ and $t\le l\le m-1$. By considering subscripts modulo $n$, we see that the possible non-zero entries in $(b_0,\ldots, b_{n-1})$ are those in position $km$, with $0\le k\le r-1$, with at least one of them non-zero.

In $B^{(l)}$, the only non-zero entries are $b_{km}^{q^l}$ in $(l+km)$-th positions, for  $0\ \le k\le r-1$.
Then the expression (\ref{eq_19}) for the $l$-th entry in $(A^tD_{\ba}B)_{(0)}$, reduces to  $a_{l-j}\sum_{k=0}^{r-1}{\alpha^{q^{l-j+km}}b_{km}^{q^l}}$, which must be zero for   $0\le j\le t-1$, $t\le l\le  m-1$  and all $\alpha\in \Fqn$. In addition, $1\le l-j\le m-1$ gives  $(a_0,a_1,\ldots,a_{m-1})=(a_0,0,\ldots,0)$, $a_0\neq 0$, and therefore  $\Aut(\Phi_{m,n,t})$ has the prescribed form.
\end{proof}

\begin{remark}
{\em
Statement i) in Theorem \ref{th_9} is obtained by taking $m=n$ in the previous Theorem.
}
\end{remark}

\begin{remark}
{\em
By Remark \ref{rem_5}, Theorem \ref{th_2} provides also the description of the automorphism group of the generalized Gabidulin code $\Phi_{m,n,t}^{(k)}$. We notice that   the  codes $\Phi_{m,n,t}^{(k)}$ are defined in different cyclic models for $\Omega_{m,n}$, for different values of $k$.
}
\end{remark}
\begin{proposition}\label{lem_6}
Let $\overline T'$ be defined as in Theorem \ref{th_2}. Then $\overline T'\simeq\End(V(n/m,q^m))$ and $T'=\overline T'\cap\GL(V')\simeq\GL(n/m,q^m)$.
\end{proposition}
\begin{proof}
Set $r=n/m$. By Proposition \ref{prop_3} we have $\End(V(r,q^m))\cong \cD_r(\Fqn)$, where  $\cD_r(\Fqn)$ is the Dickson matrix algebra of all the $q^m$-circulant $r\times r$ matrices acting on the cyclic  model $W=\{(x,x^{q^m},\ldots, x^{q^{n-m}}):x \in \Fqn\}$ for $V(r,q^m)$. Both $W$ and $V'=\{(x,x^{q},\ldots, x^{q^{n-1}})\colon x \in\Fqn\}$ are $n$-dimensional vector spaces over $\Fq$ and the map
\[
\begin{array}{cccc}
\tau: & W & \longrightarrow &V'\\
      & (x,x^{q^m},\ldots, x^{q^{n-m}}) &\mapsto &  (x,x^{q},\ldots, x^{q^{n-1}})
\end{array}
\]
is an isomorphism of vector spaces.

 A straightforward computation shows that $\tau$ induces the group isomorphism
\[
\begin{array}{cccc}
\bar\tau: & \cD_r(\Fqn) & \longrightarrow &T'\\
      & D_{(c_0,c_1,\ldots,c_{r-1})} &\mapsto &  D_{(c_0,0,\ldots0,c_1,0,\ldots,0,c_{r-1},0,\ldots,0)}
\end{array}.
\]

That is enough to get the result.
\end{proof}
\begin{corollary}
%Let $V=V(m,q)$ and $V'=V(n,q)$ with  $m>1$  any divisor of $n$. Then, 
For any given $t\in\{1,\ldots,m-1\}$, 
\[
\Aut(\Phi_{m,n,t})\simeq (\Fqm^\times\times \GL(n/m,q^m))\rtimes \Aut(\Fqm/\Fq)\rtimes \Aut(\Fq).
\]
\end{corollary}
\begin{proof}
Since   the Singer cyclic  group $S$ of $\GL(V)$  is isomorphic to the multiplicative group $\Fqm^\times$ and  the cyclic group $C$  generated by $\ell=(\ell_\xi,\ell_{\xi'})$ is isomorphic to $\Aut(\Fqm/\Fq)$,  the result follows from Theorem \ref{th_2} and Proposition \ref{lem_6}. 
\end{proof}

\section{The automorphism group of some punctured generalized twisted Gabidulin code}\label{sec_5}

\comment{
\begin{theorem}
Let $m>1$ be any divisor of $n$. For any given  $t\in\{1,\ldots,m-1\}$ and $\mu\not\in\Fp$, the punctured code  $\cH_{m,n,t,\mu,s}^{(k)}$ is  not equivalent to any generalized Gabidulin  code.
\end{theorem}
\begin{proof}
We recall that every $\theta\in\Aut(\Fq)$ defines the element $\theta:f\in\Omega_{m,n}\mapsto f^\theta\in\Omega_{m,n}$, with  $f^{\theta}(v,v')=[f(v^{\theta^{-1}},{v'}^{\theta^{-1}})]^\theta$. By Theorem \ref{th_8} for any $\theta\in \Aut(\Fq)$,    $(\id,\id;\theta)\in\Aut(\Omega_{m,n})$ is an automorphism of any generalized Gabidulin code.

 In particular, $\theta$ is an automorphism  of $\Phi_{m,n,t+1}^{(k)}=\bigoplus_{j=0}^{t}{\Omega_{j}^{(k)}}$ and it fixes every component $\Omega_i^{(k)}$, for $i=0,\ldots,t+1$.

In order to have $\theta\in\Aut(\cH_{m,n,t,\mu,s}^{(k)})$, $\theta$ must necessarily fix $\Gamma_{m,n,t,\mu,s}^{(k)}$ as a subspace of $\Omega_0^{(k)}\oplus\Omega_t^{(k)}$. This forces  $\mu\in\Fp$.
% \qed
\end{proof}
}

The following result gives  information on the geometry of the punctured code $\cH_{m,n,t,\mu,s}^{(k)}$ and it will be used to calculate the automorphism group of this  MRD code. We apply arguments similar to those used by Shekeey in \cite{she}. As we did in the previous Section, we consider only the case $k=1$ to make notation easier. The arguments below work perfectly well in the general case. We put $\Omega_{j}=\Omega_{j}^{(1)}$, $\Phi_{m,n,t-1}=\Phi_{m,n,t-1}^{(1)}$ and $\cH_{m,n,t,\mu,s}=\cH_{m,n,t,\mu,s}^{(1)}$.
\begin{theorem}\label{th_4}
%Let $n=rm$.
Let $m>1$ be any divisor of $n$ and $\mu\in\Fqn$ such that $\N_{q^n/q}(\mu)\neq (-1)^{nt}$. For any given $t\in\{1,\ldots,m-2\}$ and $s\not\equiv 0,\pm 1,\pm 2 \pmod m$, $\bigoplus_{j=1}^{t-1}\Omega_{j}$ is the unique subspace of $\cH_{m,n,t,\mu,s}$ which is equivalent to $\Phi_{m,n,t-1}$.
\end{theorem}
\begin{proof}
 Let $\varphi=(A,B;\theta)\in\Aut(\Omega_{m,n})$  such that $\Phi_{m,n,t-1}^\varphi$ is contained in $\cH_{m,n,t,\mu,s}$.
 Here, $\theta=\phi^e$. As every component $\Omega_j$ is fixed by the semilinear automorphism $\phi$, we may assume $\theta=\1$.
 We  identify the elements $A$ and $B$ with their Dickson matrices in $\cB_{m}(\Fqm)$ and $\cB_{n}(\Fqn)$, respectively. To suit our present needs, we set $A^t=D_{(a_0,a_1,\ldots,a_{m-1})}$ and $B=D_{(b_0,b_1,\ldots,b_{n-1})}^t$.

 Let $f=f_{\alpha,j}$ with $\alpha\in \Fqn$ and $0\le j\le t-2$. Let $D_{\ba}$ be the Dickson matrix of $f$ in the Singer bases $s_0,\ldots, s_{m-1}$ and $s_0',\ldots, s_{n-1}'$. Set $r=n/m$. By arguing as in the proof of Theorem \ref{th_2}, we get that  the $l$-th entry in $(A^tD_{\ba}B)_{(0)}$ is given by
\begin{equation}
\label{eq_8}
\sum_{h=0}^{m-1}{a_{h-j} \sum_{k=0}^{r-1}{\alpha^{q^{h-j+km}}b_{h-l+km}^{q^l}}},
\end{equation}
where the indices of the entries of $A$ and $B$ are taken modulo $m$ and $n$, respectively.
Since we are assuming that  $\Phi_{m,n,t-1}^\varphi$ is contained in $\cH_{m,n,t,\mu,s}$,  we must have (after substituting $h-j$ with $i$)
%\begin{equation}
\[
\sum_{i=0}^{m-1}{a_i\left(\sum_{k=0}^{r-1}{\alpha^{q^{i+km}}b_{i+j-l+km}^{q^l}}\right)}=0,
\]
%\end{equation}
for $0\le j\le t-2$, $t+1\le l\le m-1$ and all $\alpha\in \Fqn$. Therefore,
\[
a_{i}b_{i+j-l+km} = 0,\ \ \ \ \mathrm{for\ } 0\le i\le m-1.
\]

As $B\in\cB_{n}(\Fqn)$, some of the $b_i$'s are non-zero. On the other hand, the cyclic group $C=\langle \ell\rangle $ fixes every component $\Omega_{j}$. Hence,  we can assume $b_0\neq 0$ and get $a_{l-j}=0$, for  $0\le j\le t-2$ and $t+1\le l\le m-1$, i.e.
\[
a_{l} =a_{l-1}=\ldots= a_{l-t+2}= 0
\]
for $t+1\le l\le m-1$, giving $(a_0,\ldots,a_{m-1})=(a_0,a_1,a_2,0\ldots,0)$. Whenever $a_i\neq0$, we get
\[
b_{i+j-l+km} = 0,
\]
for  $0\le j\le t-2$, $0\le k\le r-1$, i.e.

\begin{equation}
\label{eq_11}
b_{i+km+1}=\ldots=b_{i+(k+1)m-3}=0
\end{equation}
for $i=0,1,2$ and $0\le k\le r-1$ since $j-l$ can take all integers from $\{1-m, 2-m, \ldots, -4,-3\}$.
We now compare the $0$-th and  $t$-th entries of $(A^tD_\ba B)_{(0)}$. From (\ref{eq_8}) we can see that the $0$-th entry of $(A D_{\ba} B^t)_{(0)}$ is
\[
\sum_{i=0,1,2}{a_i\left(\sum_{k=0}^{r-1}{\alpha^{q^{i+km}}b_{i+j+km}}\right)}
\]
and the $t$-th entry is
\[
\sum_{i=0,1,2}{a_i\left(\sum_{k=0}^{r-1}{\alpha^{q^{i+km}}b_{i+j-t+km}^{q^t}}\right)}.
\]

Since we are assuming that  $\Phi_{m,n,t-1}^\varphi$ is contained in $\cH_{m,n,t,\mu,s}$,  we must have
\[ \mu\left[\sum_{i=0,1,2}{a_i\left(\sum_{k=0}^{r-1}{\alpha^{q^{i+km}}b_{i+j+km}}\right)}\right]^{q^s}=\sum_{i=0,1,2}{a_i\left(\sum_{k=0}^{r-1}{\alpha^{q^{i+km}}b_{i+j-t+km}^{q^t}}\right)},
\]
for all $\alpha\in\Fqn$, i.e.
\[
\sum_{i=0,1,2}{a_i\left(\sum_{k=0}^{r-1}{\alpha^{q^{i+km}}b_{i+j-t+km}^{q^t}}\right)}-\mu\sum_{i=0,1,2}{a_i^{q^s}\left(\sum_{k=0}^{r-1}{\alpha^{q^{i+km+s}}b_{i+j+km}^{q^s}}\right)}=0,
\]
for all $\alpha\in\Fqn$. Since $s\neq \pm i+km$, for $i=0,1,2$ and $0\le k\le r-1$, we get
\[\{km+i \colon i=0,1,2,\, 0\leq k \leq r-1\} \cap \{km+s+i \colon i=0,1,2,\, 0\leq k \leq r-1\}=\emptyset\]
and hence
\[
\left\{\begin{array}{l}
\mu a_ib_{i+j+km}=0\\[.1in]
a_ib_{i+j-t+km}=0,
\end{array}
\right.
\]
for $0\le j\le t-2$ and $0\le k\le r-1$. Thus, whenever $a_i\neq 0$, we get
\[
\left\{\begin{array}{l}
b_{i+j+km}=0\\[.1in]
b_{i+j-t+km}=0.
\end{array}
\right.
\]
The first equation with $j=0$, the second with $j=t-2$ and (\ref{eq_11}) give us
\[
b_{i+km}=\ldots=b_{i+(k+1)m-2}=0,
\]
for $i=0,1,2$ and $0\le k\le r-1$.

For $a_0\neq 0$ we get
\[
b_{km}=\ldots=b_{(k+1)m-2}=0,
\]
for $a_1\neq 0$ we get
\[
b_{km+1}=\ldots=b_{(k+1)m-1}=0
\]
and for $a_2\neq 0$ we get
\[
b_{km+2}=\ldots=b_{(k+1)m}=0,
\]
with  $0\le k\le r-1$.

Hence, just one of the $a_i$'s is non-zero. By choosing a suitable element  in $C$ we can assume $a_0\neq 0$ so that
\[
(b_0,\ldots,b_{n-1})=(0,\ldots,0,b_{m-1},0,\ldots,0,b_{2m-1},0,\ldots,0,b_{n-1}).
\]

By recalling that $B=D_{(b_0,\ldots,b_{n-1})}^t$, we can see that  the action of $\varphi=(A,B;\1)$ on $\Omega_j^{(1)}$ is the following:
\[
\begin{array}{lcl}
f_{\alpha,j}^\varphi(v,v') %& = & f_{\alpha,j}(v^A,{v'}^B) \\[.1in]
                           & = & f_{\alpha,j}(a_0 x,b_{n-1}^q{x'}^q+b_{(r-1)m-1}^{q^{m+1}}{x'}^{q^{m+1}}+\ldots+b_{m-1}^{q^{(r-1)m+1}}{x'}^{q^{(r-1)m+1}})\\[.1in]
                          %& = & \Tr(\alpha a_0 x(b_{n-1}^q{x'}^q+b_{(r-1)m-1}^{q^{m+1}}{x'}^{q^{m+1}}+\ldots+b_{m-1}^{q^{(r-1)m+1}}{x'}^{q^{(r-1)m+1}})^{q^j}) \\[.1in]
                          & = & \Tr(\alpha a_0 x(b_{n-1}{x'}+b_{(r-1)m-1}^{q^{m}}{x'}^{q^{m}}+\ldots+b_{m-1}^{q^{(r-1)m}}{x'}^{q^{(r-1)m}})^{q^{j+1}})
\end{array}
\]
giving $f_{\alpha,j}^\varphi\in\Omega_{j+1}$. By consideration on dimensions we have
\[
\Omega_{j}^\varphi=\Omega_{j+1}
\]
giving that $\Phi_{m,n,t-1}^\varphi=\bigoplus_{j=1}^{t-1}\Omega_{j}$ is the unique subspace of $\cH_{m,n,t,\mu,s}$ which is equivalent to $\Phi_{m,n,t-1}$.
%
% \qed
\end{proof}

We are now in position to calculate the automorphism group of the MRD code $\cH_{m,n,t,\mu,s}$. %

\begin{theorem}\label{th_7}
Let $q=p^h$, $p$ a prime. Let  $m>1$ be any divisor of $n$ and $\mu\in\Fqn$ such that $\N_{q^n/q}(\mu)\neq (-1)^{nt}$. Set $r=n/m$. For any given $t\in\{1,\ldots,m-2\}$ and  $s\not\equiv 0,\pm 1,\pm 2\pmod m$, the automorphism group of $\cH_{m,n,t,\mu,s}$  is the  subgroup of $(S\times T')\rtimes C\rtimes \Aut(\Fq)$ whose elements correspond to triples $((D_{\ba},D_{\bb}); \ell^i;p^e)$,  where $\ba=(a,0,\ldots,0)$, with $a\in\Fqm$, and $\bb=(b_0,\underbrace{0,\ldots,0}_{m-1\mathrm{ \ times}},b_{m},\underbrace{0,\ldots,0}_{m-1\mathrm{ \ times}}, b_{(r-1)m},\underbrace{0,\ldots,0}_{m-1\mathrm{ \ times}})$,
with $b_{lm}\in\Fqn$ such that 
\[
\mu  a^{q^s-1} b_{lm}^{(q^s-q^t)q^{-lm}} = \mu^{p^eq^{i-lm}},
\]
for $0\le l\le r-1$ whenever $b_{lm}$ is non-zero.
\end{theorem}
\begin{proof}
From Theorem \ref{th_4} every automorphism of $\cH_{m,n,t,\mu,s}$ must fix $\bigoplus_{j=1}^{t-1}\Omega_{j}$ giving  $\Aut(\cH_{m,n,t,\mu,s})$ is a subgroup of $\Aut(\bigoplus_{j=1}^{t-1}\Omega_{j})$ which in turn is conjugate to $\Aut(\Phi_{m,n,t-1})$. By Theorem \ref{th_2},
$\Aut(\Phi_{m,n,t-1})=(S \times T')\rtimes C\rtimes\Aut(\Fq)$, and it is easy to see that this group fixes every component $\Omega_{j}$.
Let $\varphi=((A,B),\ell^i;\theta)\in\Aut(\cH_{m,n,t,\mu,s})$ with $\theta=\phi^e$. As $\varphi$ fixes every component $\Omega_{j}$, then $\varphi$ must fix $\Gamma_{m,n,t,\mu,s}$.
In addition, $\ell^i$ maps $\Gamma_{m,n,t,\mu,s}$ to $\Gamma_{m,n,t,\mu^{q^i},s}$, thus the above condition holds if and only if $ \varphi=((A,B),\1;\theta)$ maps  $\Gamma_{m,n,t,\mu^{q^i},s}$ to $\Gamma_{m,n,t,\mu,s}$.

 Let  $D_{(a,0,\ldots,0)}$ and $D_{(b_0,\ldots,0,b_m,0,\ldots,0,b_{(r-1)m},0,\ldots,0)}^t$ be the $q$-circulant  matrix of $A$ and $B$,  respectively.
Let $f=f_{\alpha,0}+f_{\mu^{q^i}\alpha^{q^s},t}$ be any bilinear form  in $\Gamma_{m,n,t,\mu^{q^i},s}$. Then, $f^{ \varphi}$ is the bilinear form defined by
\[
\begin{array}{rccl}
f^{ \varphi}(v,v')  & = &  &f_{\alpha^{p^e},0}(ax,b_0x'+b_m {x'}^{q^m}+\ldots+b_{(r-1)m}{x'}^{q^{(r-1)m}})\\[.1in]
& & + &f_{\mu^{p^eq^i}\alpha^{q^s p^e},t}(ax,b_0x'+b_m {x'}^{q^m}+\ldots+b_{(r-1)m}{x'}^{q^{(r-1)m}})\\[.1in]
& = & & \Tr(\alpha^{p^e}ax(b_0x'+b_m {x'}^{q^m}+\ldots+b_{(r-1)m}{x'}^{q^{(r-1)m}}))\\[.1in]
&  & + & \Tr(\mu^{p^eq^i}\alpha^{q^sp^e}ax(b_0x'+b_m {x'}^{q^m}+\ldots+b_{(r-1)m}{x'}^{q^{(r-1)m}})^{q^t})\\[.1in]
& = & & \Tr(a(\alpha^{p^e}b_0+\alpha^{p^eq^{(r-1)m}}b_m^{q^{(r-1)m}}+\ldots+\alpha^{p^eq^{m}}b_{(r-1)m}^{q^{m}})xx')\\[.1in]
&  & + & \Tr(a(\mu^{p^eq^i}\alpha^{p^eq^s}b_0^{q^t}+\mu^{p^eq^{(r-1)m+i}}\alpha^{p^eq^{(r-1)m+s}}b_m^{q^{(r-1)m+t}}+\ldots \\[.1in]
&  & + & \mu^{p^eq^{m+i}}\alpha^{p^eq^{m+s}}b_{(r-1)m}^{q^{m+t}})x{x'}^{q^t}).
\end{array}
\]

For $f^{ \varphi}$ to lie in $\Gamma_{m,n,t,\mu,s}$ we must have
\[
\begin{array}{l}
 a^{q^s}\mu(\alpha^{p^eq^s}b_0^{q^s}+\alpha^{p^eq^{(r-1)m+s}}b_m^{q^{(r-1)m+s}}+\ldots+\alpha^{p^eq^{m+s}}b_{(r-1)m}^{q^{m+s}})=\\[.1in]
a(\mu^{p^eq^i}\alpha^{p^eq^s}b_0^{q^t}+\mu^{p^eq^{(r-1)m+i}}\alpha^{p^eq^{(r-1)m+s}}b_m^{q^{(r-1)m+t}}+\ldots+\mu^{p^eq^{m+i}}\alpha^{p^eq^{m+s}}b_{(r-1)m}^{q^{m+t}})
\end{array}
\]
for all $\alpha\in\Fqn$. This yields
\[
\begin{array}{lcl}
\mu  a^{q^s}  b_{lm}^{q^{(r-l)m+s}} & = & a\mu^{p^eq^{(r-l)m+i}} b_{lm}^{q^{(r-l)m+t}}
\end{array}
\]
giving
\begin{equation}\label{eq_22}
\begin{array}{lcl}
\mu  a^{q^s-1} b_{lm}^{(q^s-q^t)q^{-lm}} & = & \mu^{p^eq^{i-lm}},
\end{array}
\end{equation}
for $0\le l\le r-1$ whenever $b_{lm}$ is non-zero.
% \qed
\end{proof}
\begin{remark}
{\em
Statement ii) in Theorem \ref{th_9} is obtained by taking $m=n$ in the previous Theorem.
}
\end{remark}

\begin{remark}\label{rem_8}
{\em
 By Lemma \ref{lem_5},  the   arguments used in the proof of Theorems \ref{th_7}   work perfectly well for any $k$ with $\gcd(k,n)=1$ if the cyclic model of vector  spaces and  $q$-circulant matrices involved are replaced by the $k$-th cyclic model and $q^{k}$-circulant matrices. This implies that  the automorphism group of the punctured code $\cH_{m,n,t,\mu,s}^{(k)}$ is  the subgroup of $(S\times T')\rtimes C\rtimes \Aut(\Fq)$ whose elements correspond to triples $((D_{\ba}^{(k)},D_{\bb}^{(k)}); \ell^i;p^e)$,  where $\ba=(a,0,\ldots,0)$, with $a\in\Fqm$, and $\bb=(b_0,\underbrace{0,\ldots,0}_{m-1\mathrm{ \ times}},b_{m},\underbrace{0,\ldots,0}_{m-1\mathrm{ \ times}}, b_{(r-1)m},\underbrace{0,\ldots,0}_{m-1\mathrm{ \ times}})$ with $b_{lm}\in\Fqn$ such that
 \begin{equation}\label{eq_30}
\mu  a^{q^{sk}-1} b_{lm}^{(q^{sk}-q^t)q^{-lm}} = \mu^{p^eq^{i-lm}},
\end{equation}
for $0\le l\le r-1$ whenever $b_{lm}$ is non-zero.
}
 \end{remark}

\begin{theorem}
Let $m>1$ be any divisor of $n$ and $\mu\in\Fqn$, $\mu\neq0$, such that $\N_{q^n/q}(\mu)\neq (-1)^{nt}$. For any given  $t\in\{1,\ldots,m-2\}$, 
the punctured code  $\cH_{m,n,t,\mu,s}^{(k)}$, with  $s\not\equiv 0,\pm 1,\pm 2\pmod m$  and $\gcd(n,sk-t)<m$ is  not equivalent to any generalized Gabidulin code.
\end{theorem}

\begin{proof}
In Section \ref{sec_3} we have seen that $\cH_{m,n,t,\mu,s}^{(k)}$ is an MRD $(m,n,q;m-t+1)$-code. Therefore $\cH_{m,n,t,\mu,s}^{(k)}$ has the same parameters as any generalized Gabidulin code $\cG_{(g_0,\ldots,g_{m-1});t}^{(j)}$, with $g_0,\ldots,g_{m-1}\in\Fqn$ linearly independent over $\Fq$. 

By Theorem \ref{th_8} the subgroup $L_{\cG}$ of all linear automorphisms of $\cG_{(g_0,\ldots,g_{m-1});t}^{(j)}$ is isomorphic to 
\[
(\mathbb{F}_{q^{m'}}^\times \times \GL(n/d,q^d)  )\rtimes G,
\]
for some  divisors $m'$ and  $d$ of $n$ and a subgroup $G$ of $\Aut(\Fqd/\Fq)$. Note that $m$ divides $d$. We represent the elements of $L_\cG$ by pairs of type $((a,A);\varphi)$, with $a\in\mathbb{F}_{q^{m'}}^\times$, $A\in\GL(n/d,q^d)$ and $\varphi\in G$. In particular the subgroup $\{(1,A);\id)\}$ of $L_\cG$  is isomorphic to $\GL(n/d,q^d)$. 

By way of contradiction, assume that $\cH_{m,n,t,\mu,s}^{(k)}$ is equivalent to $\cG_{(g_0,\ldots,g_{m-1});t}^{(j)}$, for some  $g_0,\ldots,g_{m-1}\in\Fqn$ linearly independent over $\Fq$. Then, $\Aut(\cH_{m,n,t,\mu,s}^{(k)})$ must be isomorphic to  $\Aut(\cG_{(g_0,\ldots,g_{m-1});t}^{(j)})$. In particular the subgroup $L_\cH$  of all linear automorphisms of  $\cH_{m,n,t,\mu,s}^{(k)}$ must be isomorphic to $L_\cG$. 
\\
Set $r=n/m$. By Theorem \ref{th_7} and Remark \ref{rem_8}, $L_\cH$ is the subgroup of $(S\times T')\rtimes C$ whose elements correspond to pairs $((D_{\ba}^{(k)},D_{\bb}^{(k)}); \ell^i)$,  where $\ba=(a,0,\ldots,0)$, with $a\in\Fqm$, and $\bb=(b_0,0,\ldots,0,b_{m},0,\ldots,0, b_{(r-1)m},0,\ldots,0)$, with $b_{lm}\in\Fqn$ satisfying Eq. (\ref{eq_30}) with $e=0$.
 \\
 For $i=0$ and $\ba=(1,0,\ldots,0)$, the pairs $((I_m,D_{\bb}^{(k)}), \id)$, with $b_{lm}\in\Fqn$ such that 
 \begin{equation}\label{eq_31}
 \mu   b_{lm}^{(q^{sk}-q^t)q^{-lm}} = \mu^{q^{-lm}},
 \end{equation}
 for $0\le l\le r-1$ whenever $b_{lm}$ is non-zero,  form a subgroup $B$ of $L_\cH$ which should be isomorphic to $\GL(n/d,q^d)$.
By raising to the $q^{lm}$-th power both sides of Eq. (\ref{eq_31}), it becomes  
 \begin{equation}\label{eq_32}
 b_{lm}^{q^t(q^{sk-t}-1)}=\mu^{1-q^{lm}}.
 \end{equation}
It is clear that the elements  $b_{lm}\in\Fqn$ that satisfy Eq. (\ref{eq_32}) corresponds to the solutions of 
\begin{equation}\label{eq_33}
X^{q^{sk-t}-1}=\mu^{1-q^{lm}}.
\end{equation}
over $\Fqn$.
If this equation has no solution in $\Fqn$ then $b_{lm}=0$. %Assume that Eq. (\ref{eq_33}) has solutions in $\Fqn$. %We now prove that they form a multiplicative coset of  $\Fqc^\times$ in $\Fqn^\times$, with $c=\gcd(n,sk-t)$. 
\\
Set $\Fqr^*=\Fqr\setminus\{0\}$. Let $x,y\in\Fqn^*$ be solutions of Eq. (\ref{eq_33}). Then $(x/y)^{q^{sk-t}-1}=1$, or equivalently $x/y\in\Fqc^*$, with $c=\gcd(n,sk-t)$. Thus the solutions of Eq. (\ref{eq_33}) over $\Fqn$ are exactly the elements in $\{\lambda x:x \in\Fqc^*\}$ where $\lambda\in\Fqn$ is a fixed solution of (\ref{eq_33}). Therefore the number of solutions of (\ref{eq_33}) is either 0 or $q^c-1$. In any case this number is strictly less than $q^m$. It follows that 
\[
|B|\le q^{cr}-1<q^{n}-1=|\GL(1,q^n)|\leq |\GL(n/d,q^d)|.
\]
\\
From the above inequality it follows that the subgroup $B$ is not isomorphic to $\GL(n/d,q^d)$. This contradicts the assumption that $\cH_{m,n,t,\mu,s}^{(k)}$ is equivalent to $\cG_{(g_0,\ldots,g_{m-1});t}^{(j)}$. The result then follows.
\end{proof}
\begin{corollary}
Let  $m>1$ be any divisor of $n$ and $\mu,\nu\in\Fqn$ such that $\N_{q^n/q}(\mu)\neq (-1)^{nt}\neq \N_{q^n/q}(\nu)$. For any given $t\in\{1,\ldots,m-2\}$ and $s\not\equiv 0,\pm 1,\pm 2\pmod m$
the punctured code $\cH_{m,n,t,\mu,s}^{(k)}$ is equivalent to $\cH_{m,n,t,\nu,u}^{(k)}$ if and only if there exist $j\in\{0,\ldots, r-1\}$, an integer $i$, a non-zero element  $a\in\Fqm$ and $b_{hm}\in\Fqn$, $h=0,\ldots, r-1$, such that $a\mu^{p^eq^{(r-h)m+i}} b_{hm}^{q^{(r-h)m+t}} =  a^{q^u} \nu b_{(h-j)m}^{q^{(j-h)m+u}}$, for $0\le h\le r-1$ (with indices of $b$ considered modulo $n$) and the $q^k$-circulant matrix $D_{\bb}$ defined by $\bb=(b_0,\underbrace{0,\ldots,0}_{m-1\mathrm{ \ times}},b_{m},\underbrace{0,\ldots,0}_{m-1\mathrm{ \ times}}, b_{(r-1)m},\underbrace{0,\ldots,0}_{m-1\mathrm{ \ times}})$ is non-singular.
\end{corollary}
\begin{proof}
We argue with $k=1$. By Theorem \ref{th_4}, $\cH_{m,n,t,\mu,s}$ and $\cH_{m,n,t,\nu,u}$ contains a unique subspace equivalent to $\Phi_{m,n,t-1}$. Therefore, any isomorphism from $\cH_{m,n,t,\mu,s}$ to $\cH_{m,n,t,\nu,u}$ is in $\Aut(\Phi_{m,n,t-1})$.  By using similar  arguments as in the proof of Theorem \ref{th_7}, we may consider isomorphisms of type  $((A,B),\1;\theta)$. Let $f=f_{\alpha,0}+f_{\mu^{q^i}\alpha^{q^s},t}$ be any bilinear form  in $\Gamma_{m,n,t,\mu^{q^i},s}$. Then, $f^{ \varphi}$ lies in $\Gamma_{m,n,t,\nu,u}$ if and only if
\[
\begin{array}{l}
 a^{q^u}\nu(\alpha^{p^eq^u}b_0^{q^u}+\alpha^{p^eq^{(r-1)m+u}}b_m^{q^{(r-1)m+u}}+\ldots+\alpha^{p^eq^{m+u}}b_{(r-1)m}^{q^{m+u}})=\\[.1in]
a(\mu^{p^eq^i}\alpha^{p^eq^s}b_0^{q^t}+\mu^{p^eq^{(r-1)m+i}}\alpha^{p^eq^{(r-1)m+s}}b_m^{q^{(r-1)m+t}}+\ldots+\mu^{p^eq^{m+i}}\alpha^{p^eq^{m+s}}b_{(r-1)m}^{q^{m+t}})
\end{array}
\]
for all $\alpha\in\Fqn$. This yields $s=jm+u$ for some $j\in\{0,\ldots, r-1\}$ giving
$a\mu^{p^eq^{(r-h)m+i}} b_{hm}^{q^{(r-h)m+t}} =  a^{q^u} \nu b_{(h-j)m}^{q^{(j-h)m+u}}$,
for $0\le h\le r-1$. Straightforward calculations show that the latter conditions  imply that $\cH_{m,n,t,\mu,s}$ is equivalent to $\cH_{m,n,t,\nu,u}$. 
% \qed
\end{proof}
\section*{Acknowledgments}
The first author is very grateful for the opportunity and the hospitality of the Department of Mathematics, Computer Science and Economics at the University of Basilicata, where he spent two weeks during the development of this research.

\noindent The authors would like to thank the referees for their  time and
useful comments that improved the first version of the paper.


\begin{thebibliography}{10}
%
\bibitem{alr} D. Augot, P. Loidreau, G. Robert, Rank metric and Gabidulin codes in characteristic zero, {\em Proceedings  ISIT 2013}, 509--513.
%
%\bibitem{bl} L.B. Beasley, T.J. Laffey, Linear operators on matrices: the invariance of rank-k matrices, {\em Linear Algebra Appl.} {\bf 133} (1990), 175--184.
%
%\bibitem{bcs} E. Ballico, A. Cossidente, A. Siciliano,  External flats to varieties in symmetric product spaces over finite fields, {\em Finite Fields Appl.} {\bf 9} (2003), 300--309.
%
\bibitem{bott} O. Bottema, On the Betti-Mathieu group, {\em Nieuw Arch. Wisk.} {\bf 16 (2)} (1930),   46--50.
%
\bibitem{br} E. Byrne, A. Ravagnani, Covering Radius of Matrix Codes Endowed with the Rank Metric, {\em SIAM J. Discrete Math.}, {\bf 31} (2017), 927--944.
%
\bibitem{car} L. Carlitz, A Note on the Betti-Mathieu group, {\em Portugaliae mathematica} {\bf 22 (3)}  (1963), 121--125.
%
\bibitem{ckww} J. de la Cruz, M. Kiermaier, A. Wassermann, W. Willems, Algebraic structures of MRD Codes, {\em Advances in Mathematics of Communications} {\bf 10} (2016), 499--510.
%
\bibitem{coop} B.N. Cooperstein, External flats to varieties in $\PG(M_{n,n}(\GF(q)))$, {\em Linear Algebra Appl.}  {\bf 267}  (1997), 175--186.
%
\bibitem{cmp} A. Cossidente, G. Marino, F. Pavese, Non-linear maximum rank distance codes, {\em Des. Codes Cryptogr.}  {\bf 79 (3)} (2016), 597-Ð609.
%
\bibitem{bcmp}  B. Csajb\'ok, G. Marino, O. Polverino, Classes and
equivalence of linear sets in $\PG(1,q^n)$,
arxiv:1607.06962
%
\bibitem{bmpz} B. Csajb\'ok, G. Marino, O. Polverino, F. Zullo, Maximum scattered linear sets and MRD-codes, {\em J. Alg. Comb.} {\bf 46} (2017), 517--531.
%
\bibitem{del}  Ph. Delsarte,  Bilinear forms over a finite field, with applications to coding theory, {\em  J. Combin. Theory Ser. A} {\bf 25} (1978),  226--241.
%
%\bibitem{d} P. Dembowski, Finite Geometries. {\em Springer 1968}.
%
\bibitem{dondur} G. Donati, N. Durante, A  generalization of the normal rational curve in $\PG(d,q^n)$  and its  associated non-linear MRD codes, {\em Des. Codes Cryptogr.}, to appear.
%
\bibitem{ds} N. Durante and A. Siciliano, Non-linear maximum rank distance codes in the cyclic model for the field reduction of finite geometries, {\em Electronic Journal of Combinatorics}  {\bf 24} (2017),  Paper 2.33, 18 pp.
%
\bibitem{fkmp} G. Faina, G. Kiss, S.  Marcugini, F. Pambianco,
The cyclic model for $\PG(n,q)$ and a construction of arcs, {\em
European J. Combin.} {\bf 23} (2002),  31--35.
%
\bibitem{gab}  E.M. Gabidulin, Theory of codes with maximum rank distance, {\em Problemy Peredachi Informatsii} {\bf 21} (1985),  3--16.
%
\bibitem{gpt2}  E.M. Gabidulin, A.V. Paramonov, O.V. Tretjakov,
Ideals over a noncommutative ring and their application in cryptology, Advances in cryptology, EUROCRYPT '91, {\em  Lecture Notes in Comput. Sci.} {\bf 547} (1991), 482--489.
%
\bibitem{gpt}  E.M. Gabidulin, A.V. Paramonov, O.V. Tretjakov, Rank errors and rank erasures correction, Proceedings of the 4th International Colloquium on Coding Theory, Dilijan, Armenia, Yerevan, 1992, pp. 11--19.
%
  \bibitem{gan} F.R. Gantmacher, The Theory of Matrices,  Vol. 1,  AMS Chelsea Publishing, Providence, RI, 1998.
%
\bibitem{h1} J.W.P. Hirschfeld, Projective Geometries Over Finite Fields, 2nd edn, Clarendon Press, Oxford, 1998.
%
%\bibitem{ht} J.W.P. Hirschfeld, J.A. Thas, {\em General Galois Geometries},
%Oxford University Press, New York, 1991.
%
 \bibitem{hu} B. Huppert, Endliche Gruppen I, Spriger, Berlin, 1967.
%
\bibitem{ln} R. Lidl, H. Niederreiter,  Finite fields.  Encyclopedia of Mathematics and its Applications, 20. Cambridge University Press, Cambridge, 1997.
 %
\bibitem{lnebe} D. Liebhold, G. Nebe,  Automorphism groups of Gabidulin-like codes,  {\em Arch. Math.} {\bf  107} (2016), 355--366.
%
\bibitem{lun} G. Lunardon, MRD-codes and linear sets, {\em  J. Combin. Theory Ser. A} {\bf  149} (2017), 1--20.
 %
\bibitem{ltz} G. Lunardon, R. Trombetti, Y. Zhou, Generalized twisted Gabidulin codes,  	arXiv:1507.07855.
 %
 \bibitem{kg} A. Kshevetskiy and E. M. Gabidulin, The new construction of rank codes. In {\em Proceedings of the Iternational Symposium on Information Theory (ISIT) 2005}, pp.  2105--2108, Sept 2005.
%
 \bibitem{kk} R. K\"otter, F. Kschischang,
Coding for errors and erasures in random network coding, {\em IEEE Trans. Inform. Theory} {\bf 54} (2008), 3579--3591.
%
\bibitem{mp}   U. Mart\'\i nez-Pe\~{n}as,On the Similarities Between Generalized Rank and Hamming Weights and Their Applications to Network Coding, {\em IEEE Trans. Inform. Th.} {\bf 62} (2016), 4081--4095. 
%
\bibitem{oo} K. Otal and F. \"Ozbudak,
Additive Rank Metric Codes, {\em IEEE Trans. Inorm. Theory} {\bf 63} (2017), 164--168.
%
\bibitem{prns} S. Puchinger, J. Rosenkilde nŽ\'e Nielsen, J. Sheekey,    Further Generalisations of Twisted Gabidulin Codes, arXiv:1703.08093 
%
 \bibitem{rav} A. Ravagnani, Rank-metric codes and their duality theory, {\em Des. Codes  Cryptogr.} {\bf 80} (2016), 197--216.
 %
 %\bibitem{roman}  S. Roman,  Field theory. Graduate Texts in Mathematics, 158. Springer, New York, 2006.
 %
 \bibitem{she} J. Sheekey, A new family of linear maximum rank distance codes, {\em Adv. Math.  Comm.} {\bf 10} (2016), 475--488.
 %
 \bibitem{sk}D. Silva, F.R. Kschischang, Universal Secure Network Coding
via Rank-Metric Codes,  {\em IEEE Trans. Inform. Theory} {\bf 57} (2011),  1124--1135.
 %
 \bibitem{skk} D. Silva, F.R.  Kschischang, R. K\"otter,  A rank-metric approach to error control in random network coding, {\em IEEE Trans. Inform. Theory} {\bf 54} (2008),  3951--3967.
%
% \bibitem{singer} J. Singer, A theorem in finite projective geometry and some applications to number theory, {\em Trans. Amer. Math. Soc.} {\bf 43} (1938), 377--385.
%
\bibitem{tsc} V. Tarokh, N. Seshadri, A.R. Calderbank,  Space-time codes for high data rate wireless communication: performance criterion and code construction, {\em IEEE Trans. Inform. Theor} {\bf 44} (1998),  744--765.
%
\bibitem{tr} A.-L. Trautmann, Isometry and automorphisms of constant dimension codes, {\em Advances in Mathematics of Communications}, {\bf 7} (2013), 147--160.
%
\bibitem{tz} R. Trombetti, Y. Zhou, Nuclei and automorphism group of generalized twisted Gabidulin codes, https://arxiv.org/pdf/1611.04447v1.pdf.
%
\bibitem{wan} Z.-X. Wan,  Geometry of matrices. In memory of Professor L. K. Hua. World Scientific Publishing Co. NJ, 1996.
%
\bibitem{wl} B. Wu, Z. Liu, Linearized polynomials over finite fields revisited, {\em Finite Fields Appl.} {\bf 22} (2013), 79--100.
%
\end{thebibliography}
 \end{document}